\documentclass[12pt,a4paper]{amsart}
\usepackage[colorlinks,citecolor=blue,urlcolor=blue,bookmarks=false,hypertexnames=true]{hyperref}
\usepackage[a4paper,margin=23mm,top=30mm]{geometry}

\usepackage{amsfonts, amsmath, amssymb, amsgen, amsthm, amscd}
\usepackage{newtxtext,newtxmath}
\usepackage[utf8]{inputenc} 

\usepackage{color}

\usepackage[all]{xy}

\usepackage{hyperref}
\usepackage{mathtools,slashed}
\usepackage{setspace}
\usepackage{enumerate}

\def\a{\alpha}

\def\d{\delta}
\def\D{\Delta}
\def\g{\gamma}

\def\Om{\Omega}
\def\s{\sigma}

\def\t{\theta}

\def\ve{\varepsilon}
\def\vp{\varphi}

\def\Id{\mathop{\rm Id}\nolimits}

\def\ot{\otimes}

\def\nb{\nabla}

\def\Db{\blacktriangledown}

\def\rt{\triangleright}

\def\dcc{\blacktriangleright\hspace{-4pt}\blacktriangleleft }
\def\cl{\blacktriangleright\hspace{-4pt} < }
\def\ccr{>\hspace{-4pt} \blacktriangleleft }

\def\D{\Delta}

\def\Om{\Omega}

\def\Id{\mathop{\rm Id}\nolimits}
\def\Im{\mathop{\rm Im}\nolimits}

\newcommand{\ps}[1]{~\hspace{-4pt}_{^{(#1)}}}
\newcommand{\pr}[1]{~\hspace{-4pt}_{^{\{#1\}}}}

\newcommand{\ns}[1]{~\hspace{-4pt}_{_{{<#1>}}}}

\newcommand{\nsb}[1]{~\hspace{-4pt}_{^{[#1]}}}

\usepackage{enumerate}

\newcommand{\G}[1]{\mathfrak{#1}}
\newcommand{\C}[1]{\mathcal{#1}}

\renewcommand{\geq}{\geqslant}

\numberwithin{equation}{section}

\newtheorem{theorem}{Theorem}[section]
\newtheorem{proposition}{Proposition}[section]

\newtheorem{corollary}[theorem]{Corollary}
\theoremstyle{definition}

\newtheorem*{remark}{Remark}

\newtheorem{definition}[theorem]{Definition}

\usepackage{geometry}
\title{Bicocycle Double Cross Constructions}

\author{O\u{g}ul Esen}
\address{Department of Mathematics, Gebze Technical University,  41400 Gebze-Kocaeli, Turkey}
\email{oesen@gtu.edu.tr}

\author{Partha Guha}
\address{Department of Mathematics,
Khalifa University\\
P.O. Box 127788, Zone -1 Abu Dhabi, UAE}
\email{partha.guha@ku.ac.ae}

\author{Serkan Sütlü}
\address{Department of Mathematics, I\c{s}ik University, 34980 \c{S}ile-\.{I}stanbul, Turkey}

\email{serkan.sutlu@isikun.edu.tr}

\begin{document}

\begin{abstract}
We introduce the notion of a bicocycle double cross product (resp. sum) Lie group (resp. Lie algebra), and a bicocycle double cross product bialgebra, generalizing the unified products. On the level of Lie groups the construction yields a Lie group on the product space of two pointed manifolds, none of which being necessarily a subgroup. On the level of Lie algebras, similarly, a Lie algebra is obtained on the direct sum of two vector spaces, none of which is required to be a subalgebra. Finally, on the quantum level the theory presents a bialgebra, on the tensor product of two (co)algebras that are not necessarily sub-bialgebras, the semidual of which being a cocycle bicrossproduct bialgebra.
\end{abstract}



\maketitle

\tableofcontents

\setlength{\parskip}{0.5cm}

\section{Introduction}

Two main problems, concerning an algebraic (or geometric, or topological) object, may be expressed as to classify its various extensions, and to seek its possible factorizations.

As was pointed out in \cite{AgorMili11,AgorMili14,AgorMili14-II} the extension problem (for groups) may be traced back to \cite{Hold1895}, and has ever since attracted much research from a variety of point of views. Each solution, then, gave rise to a novel cohomology theory classifying the type of extension in question, \cite{ChevEile48,HochSerr53-II,HochSerr53,Take81}. 

On the level of Lie groups the problem was first studied in \cite{Sche25,Sche26,Szep49,Szep50,Zapp42}, and then reconsidered in \cite{LuWe90,Maji90-II,Maji90,Take81} (and many others) under different names. As for the Lie algebra level, we refer the reader to \cite{Kac68,Mood67,Vira70}, and for relatively more recent work, to \cite{BeggMaji90,Maji90-II}. 

In an effort to unify the quantum mechanics and gravity, Hopf algebras were pointed out in \cite{Majid-book} as the objects that can offer the correct framework. These self-dual objects are widely considered as the quantum analogues of both Lie groups and Lie algebras. On this level, we refer the reader to
\cite{Majid-thesis,Sing70,Sing72}, as well as \cite{Brze97-II,BrzeHaja99,BrzeHaja09,BrzeMaji98} for an incomplete list. 

The origins of the factorization problem, on the other hand, may be found in \cite{Ore37,SzepRede50}. For more recent work, we may refer the reader to \cite{BeggGoulMaji96,Majid99,MajiRueg94}. The importance of the factorization stems from its applications in many other problems. Among such instances, we may count the role in the computations of the K-theory invariants \cite{HadfMaji07,MoscRang07,MoscRang09}, and more recently the applications in geometric mechanics (more precisely, Lagrangian and Hamiltonian dynamics) \cite{EsenGrmeGumrPave19,EsenSutlKude21,EsenSard2021,EsenSutl20,EsenSutl16,EsenSutl17,EsenSutl21}.

The extension problem and the factorization problem were recently unified and studied extensively under the title of extending structures through \cite{AgorMili11,AgorMili13,AgorMili14-III,AgorMili14,AgorMili14-II,AgorMili15,AgorMili15-II}. In the categorical language,  given a forgetful functor $F:\C{C}\to \C{D}$ with two objects $C\in \C{C}$ and $D\in \C{D}$ so that $F(C)$ being a sub-object in $D \in \C{D}$, the extending structures problem (ES problem for short) were spelled out in \cite{AgorMili11} as to describe all mathematical structures on $D$ such that $D$ becomes an object of $\C{C}$, in such a way that $C$ becomes a sub-object of $D\in\C{C}$.

Accordingly, the content of \cite{AgorMili14-II} may be summarized as the study of the ES problem corresponding to the forgetful functor $F:\C{G}r\to\C{S}et$ from the category of groups to the category of sets, while \cite{AgorMili14} considered the ES problem corresponding to $F:\G{L}ie\to\C{V}ect$ from the category of Lie algebras to the category of vector spaces. In the quantum level, on the other hand, \cite{AgorMili11} concerns the same problem associated to the functor $F:\C{H}opf \to \C{C}o\C{A}lg$ from the category of Hopf algebras to the category of coalgebras.

Put in different words; the ES problem for groups is about the extensions (resp. factorizations) of a group by a pointed set, while for Lie algebras it is the study of extensions (resp. factorizations) of a Lie algebra by a vector space, and on the quantum level the ES problem is to determine the extensions (resp. factorizations) of a Hopf algebra by a coalgebra.

By a slight abuse of language, we shall hereby reconsider the ES problem along with only a faithful functor. More precisely; in the case of groups, focusing on Lie groups only, we shall address in Proposition \ref{prop-bicocycle-double-cross-prod-gr} the extension of a pointed manifold to a Lie group by another pointed manifold. From the point of view of the factorizations, our result Proposition \ref{prop-universal-II-gr} illustrates the factorization of a Lie group into two pointed manifolds. As for the Lie algebras, we shall this time consider in Proposition \ref{prop-bicocycle-double-cross-sum} the extension of a vector space to a Lie algebra by another vector space. In other words, we shall present in Proposition \ref{prop-bicocycle-double-cross-sum-universal} the decomposition of a Lie algebra into two complementary vector spaces. Finally, in the quantum level we shall consider bialgebras, and present in Proposition \ref{prop-bicocycle-double-cross-prod} the extension of a coalgebra to a bialgebra by another coalgebra. Equivalently, we shall show in Proposition \ref{prop-universal-II} the decomposition of a bialgebra into two coalgebras. Also in this quantum level we shall address the (dual) ES problem along with the functor $F:\C{H}opf \to \C{A}lg$. Namely, in Proposition \ref{prop-cocycle-double-cross-coprod} (resp. Proposition \ref{prop-universal-cross-coprod}) we study the extension (resp. factorization) of a bialgebra by an algebra (resp. into a bialgebra and an algebra). 
 
The manuscript may be outlined as follows. 

In Section \ref{Sec-LA} we shall present the bicocycle double cross constructions on Lie algebras and Lie groups. To be more precise, in Subsection \ref{subsect-bicocycle-double-sum} we establish a Lie algebra out of two subspaces that are not necessarily Lie subalgebras. We then analyse the construction from the decomposition point of view, namely from the point of view of the decomposition of a Lie algebra into two subspaces. In Subsection \ref{subsect-Cocycle-double-cross-product-groups}, on the other hand, we study the similar construction from the point of view of Lie groups, and we present the construction of a Lie group based on two pointed submanifolds. 

The quantum counterpart of both constructions are considered in Section \ref{Sec-Quan}. In Subsection \ref{subsect-Cocycle-double-cross-product-bialgebra} we review the quantum extending structures, as well as their dual counterparts in Subsection \ref{subsect-Cocycle-double-cross-coproduct-bialgebra}. The top level of the hierarchy, concerning the extensions of quantum objects, belongs to bicocycle double cross products, which is presented in Subsection \ref{subsect-Bicocycle-double-cross-coproduct-bialgebra}.

\subsubsection*{Notation and Conventions}
 
Unadorned morphisms are assumed to be the morphisms of the ambient category, the category of vector spaces, groups, bialgebras etc. and the additional properties (such as algebra morphism, coalgebra morphism etc.) will be noted on, or under the arrows. On the level of Lie algebras we shall consider a Lie algebra $\G{g}$, with a decomposition $\G{g} \cong \G{m}\,\oplus\,\G{h}$ into vector spaces, the generic elements of which to be denoted by $\xi,\xi',\xi'',...\in \G{m}$, $Z,Z',Z'',... \in \G{h}$, and $(\xi,Z),(\xi',Z') \in \G{g}$. Similarly, we shall consider a Lie group $G$, and a decomposition $G\cong M \times H$ into manifolds. We shall, in this case, refer the generic elements as $x,x',x'',...\in M$ and $h,h',h'',... \in H$, as well as $(x,h),(x',h') \in G$. On the level of bialgebras, we shall make use of decompositions $\C{G}\cong \G{M}\ot \C{H}$ into (co)algebras, and the notations $x,x',x'',\ldots \in \G{M}$, $h,h',h'',\ldots \in \C{H}$, and $x\ot h, x'\ot h',\ldots \in \C{G}$ for the generic elements. Regarding the coalgebra structure maps, we shall use $\D$ for the comultiplication, and $\ve$ for the counit. We shall also employ the Sweedler notation for the comultiplication; namely, $\D(x) = x\ps{1} \ot x\ps{2}$, $\D(h) = h\ps{1} \ot h\ps{2}$, etc. suppressing the summations.

\section{Bicocycle Double Cross Constructions for Lie Algebras and Lie Groups} \label{Sec-LA}

In this section we shall introduce bicocycle double cross sum Lie algebras and bicocycle double cross product Lie groups. Both constructions generalize the unified products of \cite{AgorMili14} in the case of Lie algebras, and of \cite{AgorMili14-II} in the case of Lie groups.

\subsection{Bicocycle Double Cross Sum Lie Algebras}\label{subsect-bicocycle-double-sum}~

In this subsection we shall introduce the construction of a Lie algebra out of two subspaces. More precisely, given two linear spaces $\G{m}$ and $\G{h}$ we shall record the necessary and sufficient conditions for $\G{g}:=\G{m} \oplus \G{h}$ to be a Lie algebra. 

Accordingly, let $\G{m}$ and $\G{h}$ be equipped with
\begin{align}
&\phi:\G{m}\ot \G{m} \to \G{m}, \qquad \xi\ot \xi'\mapsto \phi(\xi\ot\xi') =:\phi(\xi,\xi'), \label{phi-map} \\
& \t:\G{m}\ot \G{m} \to \G{h}, \qquad \xi\ot \xi'\mapsto \t(\xi\ot\xi') =:\t(\xi,\xi'), \label{theta-map}
\end{align}
and
\begin{align}
& \mu:\G{h}\ot \G{h} \to \G{h}, \qquad Z\ot Z'\mapsto \mu(Z\ot Z') =:\mu(Z,Z'), \label{mu-map} \\
& \g:\G{h}\ot \G{h} \to \G{m}, \qquad Z\ot Z'\mapsto \g(Z\ot Z') =:\g(Z,Z'), \label{gamma-map}
\end{align}
as well as the cross-relations given by
\begin{align}
&\vp:\G{h}\ot \G{m}\to \G{m}, \qquad Z\ot \xi\mapsto \vp(Z\ot \xi)=:\vp(Z,\xi) , \label{left-action} \\
& \psi:\G{h}\ot \G{m}\to \G{h}, \qquad Z\ot \xi\mapsto\psi(Z\ot \xi) =:\psi(Z,\xi). \label{right-action}
\end{align}

\begin{proposition}\label{prop-bicocycle-double-cross-sum}
Given a pair $(\G{m},\G{h})$ of vector spaces, equipped with the linear maps \eqref{left-action}, \eqref{right-action}, \eqref{phi-map}, \eqref{theta-map}, \eqref{mu-map}, and \eqref{gamma-map}; the direct sum
$\G{g}:=\G{m}\,\oplus\,\G{h}$ is a Lie algebra via 
\begin{equation}\label{bicocycle-double-cross-sum-bracket}
[\xi+Z,\xi'+Z'] = \Big(\phi(\xi,\xi') + \vp(Z,\xi') - \vp(Z', \xi) + \g(Z,Z')\Big)+\Big(\mu(Z,Z') + \psi(Z ,\xi') - \psi(Z', \xi) +\t(\xi,\xi') \Big).
\end{equation}
if and only if 
\begin{align}
& \phi(\xi,\xi) = 0, \quad  \t(\xi,\xi) = 0, \quad \g(Z,Z)=0, \quad \mu(Z,Z) =0, \label{phi-xi-xi-theta-xi-xi-II}\\
& \vp(Z, \phi(\xi,\xi')) = \phi(\vp(Z,\xi), \xi') + \phi(\xi,\vp(Z,\xi')) + \vp(\psi(Z,\xi), \xi') - \vp(\psi(Z,\xi'),\xi) + \g(\t(\xi,\xi'), Z), \label{h-on-phi-II} \\
& \mu(Z,\t(\xi,\xi'))  = \t(\vp(Z,\xi), \xi') + \t(\xi, \vp(Z, \xi')) + \psi(\psi(Z,\xi), \xi') - \psi(\psi(Z,\xi')), \xi)  - \psi(Z, \phi(\xi,\xi')), \label{h-on-theta-and-psi-being-action-II} \\
& \vp(\mu(Z,Z'), \xi) = \vp(Z, \vp(Z',\xi)) - \vp(Z, \vp(Z',\xi)) + \g(\psi(Z,\xi),Z') + \g( Z,\psi(Z',\xi)) - \phi(\g(Z, Z'), \xi), \label{m-is-h-module-II}\\
&\psi(\mu(Z, Z') ,\xi)=  \mu(Z,\psi(Z',\xi))+\mu(\psi(Z,\xi),Z') + \psi(Z,\vp(Z',\xi)) - \psi(Z',\vp(Z,\xi)) - \t(\g(Z, Z'), \xi), \label{h-bracket-psi-comp-II} \\
& \sum_{(\xi,\xi',\xi'')}\, \phi(\phi(\xi,\xi'),\xi'') + \sum_{(\xi,\xi',\xi'')}\, \vp(\t(\xi,\xi'),\xi'') = 0, \label{phi-Jacobi-and-theta-action-II} \\
& \sum_{(\xi,\xi',\xi'')}\, \psi(\t(\xi,\xi'),\xi'') + \sum_{(\xi,\xi',\xi'')}\, \t(\phi(\xi,\xi'),\xi'') = 0, \label{cocycle-condition-for-theta-II} \\
& \sum_{(Z,Z',Z'')}\, \g(\mu(Z,Z'),Z'') - \sum_{(Z,Z',Z'')}\, \vp(Z,\g(Z',Z'')) = 0, \label{cocycle-condition-for-gamma-II} \\
& \sum_{(Z,Z',Z'')}\, \mu(\mu(Z,Z'),Z'') - \sum_{(Z,Z',Z'')}\, \psi(Z,\g(Z',Z'')) = 0, \label{mu-Jacobi-and-psi-action-II} 
\end{align}
where the summations are over the cyclic permutations of the indicated elements.
\end{proposition}

\begin{proof}
Let us begin with the anti-symmetry of the bracket \eqref{bicocycle-double-cross-sum-bracket}. We have
\[
[\xi,\xi] = 0
\]
if and only if 
\[
\phi(\xi,\xi)=0, \qquad \t(\xi,\xi) =0.
\]
Similarly we obtain
\[
[Z,Z] = 0
\]
if and only if 
\[
\g(Z,Z)=0, \qquad \mu(Z,Z) =0.
\]
As such, the bracket \eqref{bicocycle-double-cross-sum-bracket} is anti-symmetric if and only if \eqref{phi-xi-xi-theta-xi-xi-II} holds. 

As for the mixed Jacobi identities, let us first consider 
\begin{equation}\label{J-I-II}
[ [\xi, \xi'],Z] + [[\xi', Z], \xi] + [[Z, \xi], \xi'] = 0,
\end{equation}
where
\begin{align*}
 & [ [\xi, \xi'],Z] = [\phi(\xi,\xi')+ \t(\xi,\xi'),\, Z] = [\phi(\xi,\xi'), Z] + [\t(\xi,\xi'), Z] =\\
 & \Big(-\vp(Z, \phi(\xi,\xi')) + \g(\t(\xi,\xi'), Z)\Big)+\Big(\mu(\t(\xi,\xi'), Z) -\psi(Z, \phi(\xi,\xi')) \Big).
\end{align*}
Similarly,
\begin{align*}
& [[\xi', Z], \xi]  = [-\vp(Z,\xi') -\psi(Z,\xi'),\, \xi] = - [\vp(Z,\xi') , \xi] - [\psi(Z,\xi'), \xi] = \\
& =\Big( -\phi(\vp(Z,\xi'), \xi) - \vp(\psi(Z,\xi'), \xi) \Big) + \Big( - \psi(\psi(Z,\xi'), \xi) -\t(\vp(Z,\xi'), \xi)\Big),
\end{align*}
and finally
\begin{align*}
 & [[Z, \xi], \xi']  =  [\vp(Z,\xi)+ \psi(Z,\xi), \xi']  = [\vp(Z,\xi), \xi']+ [\psi(Z,\xi), \xi'] =\\
&  =\Big(\phi(\vp(Z,\xi), \xi') + \vp(\psi(Z,\xi), \xi')\Big) +\Big( \psi(\psi(Z,\xi), \xi') + \t(\vp(Z,\xi), \xi')\Big).
\end{align*}
Hence, \eqref{J-I-II} is satisfied if and only if \eqref{h-on-phi-II} and \eqref{h-on-theta-and-psi-being-action-II} hold.

Let us next consider
\begin{equation}\label{J-II-II}
[[Z, Z'], \xi] + [[Z', \xi],Z] + [[\xi, Z], Z'] = 0.  
\end{equation}
This time we have
\begin{align*}
& [[Z, Z'], \xi] = [\g(Z, Z') + \mu(Z,Z'), \xi] = [\g(Z, Z'), \xi] + [\mu(Z,Z'), \xi] = \\
& \Big(\phi(\g(Z, Z'), \xi) + \vp(\mu(Z,Z'), \xi)\Big) + \Big(\psi(\mu(Z,Z'), \xi) +\t(\g(Z, Z'), \xi)\Big)
\end{align*}
while,
\begin{align*}
& [[Z', \xi],Z]  = [\vp(Z',\xi) + \psi(Z',\xi),Z] =  [\vp(Z',\xi) ,Z]+  [\psi(Z',\xi),Z] = \\
& \Big(-\vp(Z, \vp(Z',\xi)) + \g(\psi(Z',\xi), Z)\Big) + \Big(\mu(\psi(Z',\xi), Z)-\psi(Z,\vp(Z',\xi)) \Big),
\end{align*}
and
\begin{align*}
& [[\xi, Z], Z'] =  [-\vp(Z,\xi) -\psi(Z,\xi), Z'] = - [\vp(Z,\xi) , Z'] - [\psi(Z,\xi), Z'] = \\
& \Big(\vp(Z', \vp(Z,\xi)) - \g(\psi(Z,\xi),Z')\Big) +\Big( -\mu(\psi(Z,\xi),Z') +\psi(Z',\vp(Z,\xi)) \Big).
\end{align*}
As such, \eqref{J-II-II} is satisfied if and only if \eqref{m-is-h-module-II} and \eqref{h-bracket-psi-comp-II} hold. 

We next proceed to the Jacobi identity 
\begin{equation}\label{J-III-II}
[[\xi, \xi'], \xi''] + [[\xi', \xi''], \xi] + [[\xi'', \xi], \xi'] = 0.
\end{equation}
For the first summand we have,
\begin{align*}
& [[\xi, \xi'], \xi''] = [\phi(\xi,\xi') + \t(\xi,\xi'), \xi''] = [\phi(\xi,\xi') , \xi''] + [ \t(\xi,\xi'), \xi'']  = \\
& \Big(\phi(\phi(\xi,\xi'),\xi'') + \vp(\t(\xi,\xi'),\xi'')\Big)+\Big( \psi(\t(\xi,\xi'),\xi'') + \t(\phi(\xi,\xi'),\xi'')\Big).
\end{align*}
Similarly, we observe that
\begin{align*}
 & [[\xi', \xi''], \xi] = [\phi(\xi',\xi'')+ \t(\xi',\xi''),\xi] = [\phi(\xi',\xi''),\xi]+ [\t(\xi',\xi''),\xi] = \\
& \Big(\phi(\phi(\xi',\xi''),\xi) + \vp(\t(\xi',\xi''),\xi)\Big) +\Big( \psi(\t(\xi',\xi''),\xi) + \t(\phi(\xi',\xi''),\xi)\Big).
\end{align*}
Finally,
\begin{align*}
 & [[\xi'', \xi], \xi'] = [\phi(\xi'',\xi)+ \t(\xi'',\xi),\xi'] =  [\phi(\xi'',\xi),\xi']+  [\t(\xi'',\xi),\xi'] = \\
& \Big(\phi(\phi(\xi'',\xi),\xi') + \vp(\t(\xi'',\xi),\xi')\Big)+\Big( \psi(\t(\xi'',\xi),\xi') + \t(\phi(\xi'',\xi),\xi')\Big).
\end{align*}
Accordingly, \eqref{J-III-II} is satisfied if and only if \eqref{phi-Jacobi-and-theta-action-II} and \eqref{cocycle-condition-for-theta-II} hold. 

We are finally left with the Jacobi identity
\begin{equation}\label{J-IV-II}
[[Z,Z'], Z''] + [[Z', Z''], Z] + [[Z'', Z], Z'] = 0.
\end{equation}
Just as above, we have
\begin{align*}
& [[Z, Z'], Z''] = [\g(Z,Z') + \mu(Z,Z'), Z''] = [\g(Z,Z') , Z''] +[ \mu(Z,Z'), Z'']  = \\
& \Big(-\vp(Z'',\g(Z,Z')) + \g(\mu(Z,Z') , Z'')\Big)+\Big( -\psi(Z'',\g(Z,Z')) + \mu(\mu(Z,Z') , Z'')\Big),
\end{align*}
where similarly
\begin{align*}
& [[Z', Z''], Z] = [\g(Z',Z'') + \mu(Z',Z''), Z] = [\g(Z',Z'') , Z] +[ \mu(Z',Z''), Z]  = \\
& \Big(-\vp(Z,\g(Z',Z'')) + \g(\mu(Z',Z'') , Z)\Big)+\Big( -\psi(Z,\g(Z',Z'')) + \mu(\mu(Z',Z'') , Z)\Big),
\end{align*}
and
\begin{align*}
& [[Z'', Z], Z'] = [\g(Z'',Z) + \mu(Z'',Z), Z'] = [\g(Z'',Z) , Z'] +[ \mu(Z'',Z), Z']  = \\
& \Big(-\vp(Z',\g(Z'',Z)) + \g(\mu(Z'',Z) , Z')\Big)+\Big( -\psi(Z',\g(Z'',Z)) + \mu(\mu(Z'',Z) , Z')\Big).
\end{align*}
Accordingly, \eqref{J-IV-II} is satisfied if and only if \eqref{cocycle-condition-for-gamma-II} and \eqref{mu-Jacobi-and-psi-action-II} hold.
\end{proof}

We shall call the pair $(\G{m},\G{h})$ of vector spaces satisfying \eqref{phi-xi-xi-theta-xi-xi-II} - \eqref{mu-Jacobi-and-psi-action-II} a \emph{bicocycle matched pair}, and we shall employ the notation $\G{m}{\,}_{\g\hspace{-0.1cm}}\bowtie_\t \G{h} := \G{m}\oplus \G{h}$ for the \emph{bicocycle double cross sum} Lie algebra of Proposition \ref{prop-bicocycle-double-cross-sum}.

From the decomposition point of view, we have the following generalization of \cite[Prop. 8.3.2]{Majid-book} and \cite[Thm. 3.4]{AgorMili14}.

\begin{proposition}\label{prop-bicocycle-double-cross-sum-universal}
Any Lie algebra $\G{g}$ with two complementary subspaces $\G{m},\G{h} \subseteq \G{g}$ is isomorphic, as Lie algebras, to the bicocycle double cross sum of these subspaces; that is, $\G{g}\cong \G{m}{\,}_{\g\hspace{-0.1cm}}\bowtie_\t \G{h}$, where the linear maps \eqref{left-action}, \eqref{right-action}, \eqref{phi-map}, \eqref{theta-map}, \eqref{mu-map}, and \eqref{gamma-map} may be recovered from
\[
[\xi,\xi'] = \phi(\xi,\xi') + \t(\xi,\xi'), \quad [Z,\xi] = \vp(Z, \xi) + \psi(Z , \xi), \quad [Z,Z'] = \g(Z,Z') + \mu(Z,Z').
\]
\end{proposition}

A few comments are in order.

\begin{remark}
This time both $\G{m}$ and $\G{h}$ are subspaces. Accordingly, neither \eqref{right-action} is expected to be a (right) action, nor \eqref{left-action} is a priori (left) action. In view of \eqref{h-on-theta-and-psi-being-action-II}, \eqref{right-action} being a right action is equivalent to the (left) \emph{adjoint} action of $\G{h}$ on $\Im\t$ being given by \emph{derivations}; namely
\[
\mu(Z,\t(\xi,\xi'))  = \t(\vp(Z,\xi), \xi') + \t(\xi, \vp(Z, \xi')).
\]
Similarly, in view of \eqref{m-is-h-module-II}, \eqref{left-action} being a left action is equivalent to the (right) adjoint action of $\G{m}$ on $\Im\g$ being given by derivations; that is,
\[
\phi(\g(Z, Z'), \xi) = \g(\psi(Z,\xi),Z') + \g( Z,\psi(Z',\xi)).
\]
On the other hand, neither \eqref{phi-map}, nor \eqref{mu-map} is a priori a Lie bracket. Indeed, \eqref{phi-Jacobi-and-theta-action-II} reveals that \eqref{phi-map} satisfying the Jacobi identity is equivalent to 
\[
\sum_{(\xi,\xi',\xi'')}\, \vp(\t(\xi,\xi'),\xi'' )= 0.
\]
For instance, if \eqref{theta-map} or \eqref{left-action} is trivial, then \eqref{phi-map} determines a Lie bracket on $\G{m}$. Quite similarly, it follows from \eqref{mu-Jacobi-and-psi-action-II} that \eqref{mu-map} satisfying the Jacobi identity is equivalent to 
\[
\sum_{(Z,Z',Z'')}\, \psi(Z,\g(Z',Z'')) = 0 .
\]
As such, if \eqref{gamma-map} or \eqref{right-action} is trivial, then \eqref{mu-map} defines a Lie algebra bracket on $\G{h}$. We do note also that if  \eqref{phi-map} is a Lie bracket and \eqref{right-action} is a right action of $\G{m}$ on $\G{h}$, the condition \eqref{cocycle-condition-for-theta-II} can simply be expressed as
\[
\t\in H^2(\G{m},\G{h}),
\]
that is, \eqref{theta-map} is a 2-cocycle in the Lie algebra cohomology of $\G{m}$, with coefficients in $\G{h}$; with the right $\G{m}$-module structure on $\G{h}$ being given by \eqref{right-action}. If, furthermore, \eqref{left-action},\eqref{gamma-map}, and \eqref{mu-map} are all trivial, then the bicocycle double cross sum $\G{m}{\,}_{\g\hspace{-0.1cm}}\bowtie_\t \G{h}$ becomes $\G{m}\ltimes_\t\G{h}$; the abelian extension of $\G{m}$ by $\G{h}$. Symmetrically, if  \eqref{mu-map} is a Lie bracket and \eqref{left-action} determines a left action of $\G{h}$ on $\G{m}$, then the condition \eqref{cocycle-condition-for-gamma-II} turns out to 
\[
\g\in H^2(\G{h},\G{m}),
\]
that is, \eqref{gamma-map} is a 2-cocycle in the Lie algebra cohomology of $\G{h}$, with coefficients in $\G{m}$; with the left $\G{h}$-module structure on $\G{m}$ being the one given by \eqref{left-action}. If, furthermore, \eqref{right-action},\eqref{theta-map}, and \eqref{phi-map} are all trivial, then the bicocycle double cross sum $\G{m}{\,}_{\g\hspace{-0.1cm}}\bowtie_\t \G{h}$ becomes $\G{m}{\,}_{\g\hspace{-0.1cm}}\rtimes \G{h}$; the abelian extension of $\G{h}$ by $\G{m}$.

Finally, the bicocycle double cross sum $\G{m}{\,}_{\g\hspace{-0.1cm}}\bowtie_\t \G{h} $ encompasses two types of unified products under one roof; in case \eqref{theta-map} is trivial, then $\G{m}{\,}_{\g\hspace{-0.1cm}}\bowtie_\t \G{h} = \G{m}{\,}_{\g\hspace{-0.1cm}}\bowtie \G{h}$ is precisely the (left-handed) unified product of \cite[Thm. 3.2]{AgorMili14}, while if \eqref{gamma-map} is trivial, then $\G{m}{\,}_{\g\hspace{-0.1cm}}\bowtie_\t \G{h} = \G{m}\bowtie_\t \G{h}$ is its right-handed counterpart.
\end{remark}

We shall conclude with an illustration.

\textit{The Lie algebra of formal vector fields on the line.}~

Let $W_1$ be the Lie algebra of formal vector fields on the line, \cite{Fuks-book,Gonc73}, which is an infinite dimensional Lie algebra with a basis $z_i:=x^{i+1}\partial/\partial x$, for $i\geq -1$, and the Jacobi-Lie bracket of vector fields that corresponds to 
\begin{equation} \label{grad-fvf}
[z_i,z_j]=(j-i)z_{i+j}, \qquad i,j \geq -1.
\end{equation}
Let also
\[
\G{m} := \langle z_\ell \mid \ell=4k,\text{ or } \ell=4k+1,\, k\geq 0\rangle, \qquad \G{h} := \langle z_\ell \mid \ell=4k-1,\text{ or } \ell=4k+2,\, k\geq 0\rangle.
\] 
It is then evident that neither $\G{m}$, nor $\G{h}$ is a subalgebra of $W_1$. Nevertheless, they are both subspaces, and $\G{m}\oplus \G{h} \cong W_1$. As such,
\[
W_1 \cong \G{m} \,{}_\g\bowtie_\t \G{h}, 
\]
where the structure maps are given by
\begin{align*}
&\phi:\G{m}\otimes \G{m} \to \G{m},\qquad  \phi(z_{4k},z_{4t})=4(t-k)z_{4(k+t)}, \quad \phi(z_{4k},z_{4t+1})=(4(t-k)+1)z_{4(k+t)+1}, \\
&\mu:\G{h}\otimes \G{h} \to \G{h},\qquad  \mu(z_{4k-1},z_{4t-1})=\begin{cases}
4(t-k)z_{4(k+t-1)+2} & \text{ if } t+k\geq 1, \\
0 & \text{ if } t=k=0,
\end{cases} \\
&\theta:\G{m}\otimes \G{m} \to \G{h}, \qquad \theta(z_{4k+1},z_{4t+1})=4(t-k)z_{4(k+t)+2}, \\
&\g:\G{h}\otimes \G{h} \to \G{m}, \qquad \g(z_{4k-1},z_{4t+2})=(4(t-k)+3)z_{4(k+t)+1},\quad  \g(z_{4k+2},z_{4t+2})=4(t-k)z_{4(k+t+1)},\\
&\vp:\G{h}\otimes \G{m} \to \G{m}, \qquad  z_{4k-1}\rt z_{4t+1}=(4(t-k)+2)z_{4(t+k)},\\
&\psi:\G{h}\otimes \G{m} \to \G{h}, \qquad \psi(z_{4k-1},z_{4t})=(4(t-k)+1)z_{4(t+k)-1}, \quad \psi(z_{4k+2},z_{4t})=(4(t-k)-2)z_{4(t+k)+2}, \\
&\hspace{3.5cm} \psi(z_{4k+2},z_{4t+1})=(4(t-k)-1)z_{4(t+k+1)-1}.
\end{align*}

\subsection{Bicocycle Double Cross Product Lie Groups}\label{subsect-Cocycle-double-cross-product-groups}~

In the present subsection we shall develop a theory for the Lie groups of bicocycle double cross sum Lie algebras. In other words, we shall now introduce a construction that allows to construct a Lie group over two manifolds that are not necessarily subgroups. From the decomposition point of view, we shall present a theory that allows to decompose a Lie group into two submanifolds, none of which being necessarily a subgroup. The construction we present here generalizes a double cross product by two (twisted) 2-cocycles, and a unified product by one (twisted) 2-cocycle.

\begin{proposition}\label{prop-bicocycle-double-cross-prod-gr}
Let $(M,H)$ be a pair of two (pointed) manifolds; namely $(M,e)$ and $(H,1)$, equipped with the (smooth) maps 
\begin{align}
& \vp:H\times M \to M, \qquad (h, x)\mapsto \vp(h,x), \label{left-action-gr-II} \\
& \psi:H\times M \to H, \qquad (h, x)\mapsto \psi(h, x),\label{right-action-gr-II} \\
& \phi:M\times M \to M, \qquad (x, x')\mapsto \phi(x,x')=:x \cdot x', \label{phi-map-gr-II} \\
& \t:M\times M \to H, \qquad (x, x')\mapsto \t(x, x') , \label{theta-map-gr-II}, \\
& \mu:H\times H \to H, \qquad (h,h')\mapsto \mu(h,h')=:h\ast h' \label{mu-map-gr-II} \\
& \g:H\times H \to M, \qquad (h,h')\mapsto \g(h,h') \label{gamma-map-gr-II} 
\end{align}
that satisfy
\begin{align}\label{normalization-gr-II}
\begin{split}
& \vp(1,x) = x, \qquad \vp(h,e) = e, \\
& \psi(h,e) = h, \qquad \psi(1,x) = 1, \\
& e \cdot e = e, \qquad 1\ast 1 = 1, \\
& \t(e,e) = 1, \qquad \g(1,1) = e.
\end{split}
\end{align}
Then, $M\times H$ is a Lie group with the multiplication 
\begin{equation}\label{bicocycle-double-cross-prod-multp-gr}
(x,h)(x',h') = \Big(x\cdot [\vp(h, x' ) \cdot\g(\psi(h,x'),h')],\,[\t(x,\vp(h, x'))\ast \psi(h ,x')] \ast h' \Big),
\end{equation}
and the unit $(e,1) \in M\times H$ if and only if 
\begin{align}
& e\cdot x = x = x\cdot e, \qquad 1\ast h = h = h \ast 1, \label{phi-mu-map-gr-II} \\
& \t(x,e) = 1 = \t(e,x), \qquad \g(h,1) = e = \g(1,h), \label{theta-gamma-map-gr-II} \\
& \vp(h, x'\cdot x'')\cdot \g(\psi(h,x'\cdot x''),\t(x',x'')\ast h'') = \label{left-action-gr-on-multp-III}\\
& \hspace{3cm} \vp(h,x') \cdot [\vp(\psi(h,x'),x'')\cdot \g(\psi(\psi(h,x'),x''), h'')], \notag \\
& \psi(h,x'\cdot x'') \ast [\t(x', x'')\ast h''] = [\t(\vp(h, x'), \vp(\psi(h, x'), x''))\ast\psi(\psi(h,x'), x'')]\ast h'', \label{right-action-theta-comp-gr-III} \\
& \t(x\cdot\g(h,h'), \vp(h\ast h', x''))\ast\psi(h\ast h', x'') = \label{right-action-gr-on-multp-III} \\
& \hspace{3cm} [\t(x, \vp(h,\vp(h',x'')))\ast \psi(h,\vp(h',x''))]\ast \psi(h', x''),  \notag \\
& [x\cdot\g(h,h')]\cdot\vp(h\ast h', x'') = x\cdot [\vp(h,\vp(h', x''))\cdot \g(\psi(h,\vp(h',x'')), \psi(h',x''))], \label{M-is-left-H-module-III-gr} \\
& x\cdot (x'\cdot x'') = (x\cdot x')\cdot [\vp(\t(x, x'), x'') \cdot \g(\psi(\t(x, x'), x''), h'')], \label{phi-map-gr-assoc-III} \\
& \t(x, x'\cdot x'')\ast [\t(x', x'')\ast h''] =  [\t(x\cdot x',\vp(\t(x, x'), x''))\ast \psi(\t(x, x'), x'')] \ast h'',  \label{theta-map-gr-cocycle-III} \\
& [x\cdot\g(h,h')] \cdot \g(h\ast h',h'') = x\cdot [\vp(h,\g(h',h''))\cdot \g(\psi(h,\g(h',h'')), h'\ast h'')], \label{gamma-map-gr-cocycle-III}\\
& (h\ast h') \ast h'' = [\t(x, \vp(h,\g(h',h'')))\ast \psi(h,\g(h',h''))]\ast (h'\ast h''), \label{gamma-map-gr-assoc-III} 
\end{align}
for any $x,x',x'' \in M$, and any $h,h',h''\in H$, and furthermore
\begin{align}
& \text{for any } (x,h)\in\G{M}\times \C{H} \text{ there is } (x^r, h^r)\in\G{M} \times \C{H} \text{ such that } x\cdot x^r = e, \text{ and } h\ast h^r = 1, \label{right-inverses-gr} \\
& \text{for any } (x,h)\in\G{M}\times \C{H} \text{ there is } (x^\ell, h^\ell)\in\G{M} \times \C{H} \text{ such that } x^\ell\cdot x = e, \text{ and } h^\ell\ast h= 1. \label{left-inverses-gr} 
\end{align}
\end{proposition}

\begin{proof}
Let us begin with the identity element. On one hand we have
\[
(e, 1)(x, 1) = \Big(e \cdot [x \cdot e],\, [\t(e,x) \ast 1]\ast 1\Big),
\]
while on the other hand
\[
(x, 1)(e, 1) = \Big(x \cdot e^2,\, [\t(x,e) \ast 1]\ast 1\Big).
\]
Accordingly, 
\[
(e,1)(x, 1) = (x, 1) = (x, 1)(e, 1)
\]
for any $x\in M$ if and only if 
\[
e\cdot (x \cdot e) = x = x \cdot e^2
\]
and
\[
[\t(e,x) \ast 1]\ast 1 = 1 = [\t(x,e) \ast 1]\ast 1.
\]
On the other hand, 
\[
(e, h)(e, 1) = \Big( e \cdot (e \cdot \g(h,1)),\, (1 \ast h) \ast 1 \Big)
\]
while
\[
(e, 1)(e, h) = \Big(e \cdot (e \cdot \g(1,h)),\, 1^2 \ast h\Big).
\]
Thus, 
\[
(e, h)(e, 1) = (e, h) = (e, 1)(e, h)
\]
for any $h\in \C{H}$, if and only if
\[
1^2\ast h = h = (1 \ast h) \ast 1,
\]
and
\[
 e \cdot (e \cdot \g(h,1)) = e =  e \cdot (e \cdot \g(1,h)).
\]
As a result, in view of the last line of \eqref{normalization-gr-II}, $(e,1) \in M\times H$ is the identity element of the multiplication \eqref{bicocycle-double-cross-prod-multp-gr} if and only if \eqref{phi-mu-map-gr-II} and \eqref{theta-gamma-map-gr-II} hold.

Let us next consider the (mixed) associativity conditions for \eqref{bicocycle-double-cross-prod-multp-gr}.

As for $(x, 1), (x', 1), (x'',h'')\in M\times H$, we have on one hand
\begin{align*}
& (x, 1)[(x', 1)(x'', h'')] = (x,1)\Big(x'\cdot x'' ,\, \t(x', x'')\ast h''\Big) = \\
& \Big(x\cdot [x'\cdot x''], \, \t(x,x'\cdot x'')\ast [\t(x', x'')\ast h'']\Big),
\end{align*}
while on the other hand,
\begin{align*}
&[(x, 1)(x', 1)](x'', h'') = \Big(x\cdot x', \t(x, x')\Big)(x'', h'') = \\
&  \Big([x\cdot x'] \cdot [\vp(\t(x, x'), x'') \cdot \g(\psi(\t(x, x'), x''), h'')], \, [\t(x\cdot x',\vp(\t(x, x'), x''))\ast \psi(\t(x, x'), x'')] \ast h''\Big).
\end{align*}
Accordingly,
\[
 (x, 1)[(x', 1)(x'', h'')]   =  [(x, 1)(x', 1)](x'', h'') 
\]
if and only if \eqref{phi-map-gr-assoc-III} and \eqref{theta-map-gr-cocycle-III} are satisfied. 

We, then, proceed onto the mixed associativity condition for $(x, h), (e, h'), (e,h'')\in M\times H$. To this end, we have
\begin{align*}
&[(x, h)(e, h')](e, h'') = \Big(x\cdot\g(h,h'), h\ast h'\Big)(e, h'') = \\
&  \Big([x\cdot\g(h,h')]\cdot \g(h\ast h', h''),\, [h\ast h']\ast h''\Big),
\end{align*}
and
\begin{align*}
& (x, h)[(e, h')(e, h'')] = (x, h)\Big(\g(h',h''), h'\ast h''\Big) = \\
& \Big(x \cdot [\vp(h,\g(h',h''))\cdot \g(\psi(h,\g(h',h'')), h'\ast h'')], \,  [\t(x, \vp(h,\g(h',h'')))\ast \psi(h,\g(h',h''))]\ast [h'\ast h'']\Big).
\end{align*}
Hence,
\[
[(x, h)(e, h')](e, h'')  = (x, h)[(e, h')(e, h'')] 
\]
if and only if  \eqref{gamma-map-gr-cocycle-III} and \eqref{gamma-map-gr-assoc-III} hold. 

Next, we observe for $(x, h), (e, h'), (x'',1)\in M\times H$ that
\begin{align*}
& [(x, h)(e, h')](x'',1) = \Big(x\cdot \g(h,h'), h\ast h'\Big)(x'',1) = \\
& \Big([x\cdot \g(h,h')] \cdot \vp(h\ast h', x''), \, \t(x\cdot \g(h,h'), \vp(h\ast h', x'')) \ast \psi(h\ast h', x'')\Big),
\end{align*}
and that
\begin{align*}
& (x, h)[(e,h')(x'', 1)] =(x,h)\Big(\vp(h',x''), \psi(h', x'')\Big) = \\
& \Big(x\cdot [\vp(h,\vp(h',x''))\cdot \g(\psi(h,\vp(h',x'')),\psi(h', x''))] , \, [\t(x, \vp(h,\vp(h',x'')))\ast \psi(h,\vp(h',x''))]\ast \psi(h', x'')\Big).
\end{align*}
Therefore, we conclude that 
\[
[(x, h)(e, h')](x'',1) = (x, h)[(e,h')(x'', 1)] 
\]
if and only if \eqref{right-action-gr-on-multp-III} and \eqref{M-is-left-H-module-III-gr} are satisfied.

Lastly, for $(e, h), (x',1), (x'',h'')\in M\times H$ we consider the associativity through
\begin{align*}
& (e , h)[(x' ,1)(x'' , h'')] = (e, h)\Big(x'\cdot x'', \t(x',x'')\ast h''\Big) = \\
&\Big(\vp(h,x'\cdot x'')\cdot \g(\psi(h,x'\cdot x''),\t(x',x'')\ast h''), \, \psi(h,x'\cdot x'') \ast [\t(x',x'')\ast h'']\Big)
\end{align*}
and
\begin{align*}
& [(e , h)(x' ,1)](x'' , h'') = \Big(\vp(h,x'), \psi(h,x')\Big)(x'' , h'') = \\
& \Big(\vp(h,x') \cdot [\vp(\psi(h,x'),x'')\cdot \g(\psi(\psi(h,x'),x''), h'')], [\t(\vp(h,x'), \vp(\psi(h,x'),x''))\ast \psi(\psi(h,x'),x'')]\ast h''\Big),
\end{align*}
and we conclude that 
\[
(e , h)[(x' ,1)(x'' , h'')]=[(e , h)(x' ,1)](x'' , h'')  
\]
if and only if \eqref{left-action-gr-on-multp-III} and \eqref{right-action-theta-comp-gr-III} hold.

We have come to the inverses now. Given any $(x,1) \in \G{M}\times \C{H}$, there is $(x',h') \in \C{M}\times \C{H}$ such that 
\begin{equation}\label{right-inverse-MH}
(x,1)(x',h') = (e,1)
\end{equation}
if and only if \eqref{right-inverses-gr} holds. We note, in this case, that 
\[
x' = x^r, \qquad h' = \t(x,x^r)^r.
\]
Similarly, given $(e,h) \in \G{M}\times \C{H}$, there is $(x'',h'') \in \C{M}\times \C{H}$ such that 
\begin{equation}\label{left-inverse-MH}
(x'',h'')(e,h) = (e,1)
\end{equation}
if and only if \eqref{left-inverses-gr} holds. In this case, 
\[
x'' = \g(h^\ell,h)^\ell, \qquad h'' = h^\ell.
\]
\end{proof}

We shall denote the Lie group of Proposition \ref{prop-bicocycle-double-cross-prod-gr} by $M {\,}_{\g\hspace{-0.1cm}}\bowtie_\t H :=M\times H$, and we shall call it the \emph{bicocycle double cross product} of the manifolds $M$ and $H$.

A series of remarks are in order.

\begin{remark}
Given $(x,1) \in \G{M}\times \C{H}$,
\begin{align*}
& (x,1)[(x^r,h^r)(x,1)] = (x,1)\Big\{(x^r,1)[(e,h^r)(x,1)]\Big\} = [(x,1)(x^r,1)][(e,h^r)(x,1)] = \\
& \Big\{[(x,1)(x^r,1)](e,h^r)\Big\}(x,1) = [(x,1)(x^r,h^r)](x,1) = (x,1),
\end{align*}
where the first and the third equalities follow from \eqref{right-action-gr-on-multp-III} and \eqref{M-is-left-H-module-III-gr}, while the second and the forth equalities are results of \eqref{phi-map-gr-assoc-III} and \eqref{theta-map-gr-cocycle-III}. Therefore, we have
\[
(x,1)^{-1} = (x^r,\t(x,x^r)^r).
\]
Similarly,
\begin{align*}
& [(e,h)(x^\ell,h^\ell)](e,h) =  \Big\{[(e,h)(x^\ell,1)](e, h^\ell)\Big\}(e,h) = [(e,h)(x^\ell,1)][(e, h^\ell)(e,h)] = \\
& (e,h)\Big\{(x^\ell,1)[(e, h^\ell)(e,h)]\Big\} = (e,h)[(x^\ell,h^\ell)(e,h)] = (e,h),
\end{align*}
where the first and the third equalities follow from \eqref{left-action-gr-on-multp-III} and \eqref{right-action-theta-comp-gr-III}, while the second and the forth equalities are results of \eqref{gamma-map-gr-cocycle-III} and \eqref{gamma-map-gr-assoc-III}. Accordingly,
\[
(e,h)^{-1} = (\g(h^\ell,h)^\ell, h^\ell).
\]
As such, the inversion is given as
\begin{align*}
& (x,h)^{-1} = [(x,1)(e,h)]^{-1} = (e,h)^{-1}(x,1)^{-1} = (x^r,\t(x,x^r)^r)(\g(h^\ell,h)^\ell, h^\ell) = \\
& \Big(x^r\cdot [\vp(\t(x,x^r)^r, \g(h^\ell,h)^\ell) \cdot \g(\psi(\t(x,x^r)^r, \g(h^\ell,h)^\ell), h^\ell)], \\
& \hspace{3cm} [\t(x^r, \vp(\t(x,x^r)^r, \g(h^\ell,h)^\ell))\ast \psi(\t(x,x^r)^r, \g(h^\ell,h)^\ell)]\ast h^\ell\Big).
\end{align*}
\end{remark}

\begin{remark}
If the map \eqref{theta-map-gr-II} is trivial, and \eqref{right-action-gr-II} is a trivial right action, then the group $M {\,}_{\g\hspace{-0.1cm}}\bowtie_\theta H = M {\,}_{\g\hspace{-0.1cm}}\bowtie H$ in Proposition \ref{prop-bicocycle-double-cross-prod-gr} coincides with the one in \cite[Ex. 6.3.3]{Majid-book}. If, on the other hand, the map \eqref{gamma-map-gr-II} is trivial, and \eqref{left-action-gr-II} is a trivial left action, then the algebra $M {\,}_{\g\hspace{-0.1cm}}\bowtie_\theta H = M \bowtie_\theta H$ becomes the right-handed version of the group in \cite[Ex. 6.3.3]{Majid-book}. If, furthermore, $H$ is commutative, then \eqref{right-action-gr-II} is a right action. In this case, \eqref{theta-map-gr-cocycle-III} indicates that \eqref{theta-map-gr-II} is a 2-cocycle in the group cohomology of $M$ with coefficients in $H$, that is,
\[
\t\in H^2(M,H).
\]
Finally, if \eqref{right-action-gr-II} and \eqref{theta-map-gr-II} are trivial, and $M$ is commutative, then \eqref{left-action-gr-II} turns out to be a left action. In this case, \eqref{gamma-map-gr-cocycle-III} implies that \eqref{gamma-map-gr-II} is a 2-cocycle in the group cohomology of $H$ with coefficients in $M$, that is,
\[
\g\in H^2(H,M).
\]
We note also that if \eqref{theta-map-gr-II} is trivial, then $M {\,}_{\g\hspace{-0.1cm}}\bowtie_\t H = M {\,}_{\g\hspace{-0.1cm}}\bowtie H$ is precisely the unified product in \cite[Thm. 3.5]{AgorMili14-II}, whereas if \eqref{gamma-map-gr-II} is trivial, then $M {\,}_{\g\hspace{-0.1cm}}\bowtie_\t H= M \bowtie_\t  H$ is its right-handed analogue.
\end{remark}

As in the preceding subsection, from the decomposition point of view we have the following generalization of \cite[Thm. 7.2.3]{Majid-book} and \cite[Thm. 3.1]{AgorMili14-II}, compare with \cite[Prop. 2.3]{BespDrab01}. 

\begin{proposition}\label{prop-universal-II-gr}
Given two manifolds $M$ and $H$, a Lie group $G$, and the (smooth) maps
\[
\xymatrix{
M \ar@{^{(}->}[r]_i  & G    &  \ar@{_{(}->}[l]^j H,
}
\]
if $m\circ (i\times j):M\times H \to G$ is a diffeomorphism, where $m:G\times G\to G$ denotes the multiplication in $G$, then $G\cong M {\,}_{\g\hspace{-0.1cm}}\bowtie_\t H$ as Lie groups, where the maps \eqref{left-action-gr-II}, \eqref{right-action-gr-II}, \eqref{phi-map-gr-II}, \eqref{theta-map-gr-II}, \eqref{mu-map-gr-II}, \eqref{gamma-map-gr-II} may be obtained from
\[
hx = \vp(h, x) \psi(h,x), \qquad xx'=(x\cdot x') \t(x,x'), \qquad hh' = \g(h, h') (h\ast h'),
\]
for any $x,x' \in M$, and any $h,h'\in H$.
\end{proposition}

\begin{proof}
Let us first consider the maps
\begin{equation}\label{f-map-II-gr}
f:H\times M \to M\times H
\end{equation}
given by
\[
j(h)i(x) = \Big(m\circ (i\times j) \circ f \Big)(h, x),
\]
and  
\begin{equation}\label{g-map-II-gr}
g:M\times M\to M\times H
\end{equation}
given by
\[
i(x)i(x') = \Big(m\circ (i\times j) \circ g \Big)(x, x'),
\]
together with
\begin{equation}\label{r-map-II-gr}
r:H\times H \to M\times H
\end{equation}
given by
\[
j(h)j(h') = \Big(m\circ (i\times j) \circ r \Big)(h\times h').
\]
Accordingly, we define 
\begin{align}\label{maps-I-II-gr}
\begin{split}
& \vp:H\times M\to M, \qquad \vp:= \pi_1\circ f, \\
& \psi:H\times M\to H, \qquad \psi:= \pi_2\circ f, \\
& \phi:M\times M\to M\times M, \qquad \phi:= \pi_1\circ g, \\
& \t:M\times M\to H, \qquad \t:= \pi_2\circ g, \\
& \mu:H\times H \to H, \qquad \mu:= \pi_2\circ r, \\
& \g:H\times H \to M, \qquad \g:= \pi_1\circ r.
\end{split}
\end{align}
More explicitly,
\begin{align*}
& f(h\ot x) = \vp(h,x) \ot \psi(h,x), \\
& g(x\ot x') = \phi(x,x') \ot \t(x,x'), \\
& r(h\ot h') = \g(h,h') \ot \mu(h,h').
\end{align*}
Therefore, the multiplication on $M\times H$ which is borrowed from $G$ is given by
\begin{align*}
& hx := \vp(h, x) \psi(h,x), \\
& xx':= \phi(x, x') \t(x,x'), \\
& hh' := \g(h, h')\mu(h,h'),
\end{align*}
for any $x,x'\in M$, and any $h,h'\in H$, with the unit being $(e,1)\in M\times H$. 

To conclude, we impose the associativity and the unitality of the group structure on $M\times H$ to derive \eqref{normalization-gr-II}, while \eqref{phi-mu-map-gr-II}-\eqref{gamma-map-gr-assoc-III} are all satisfied by the maps of \eqref{maps-I-II-gr}. 
\end{proof}

Aside from the Lie groups of bicocycle double cross sum Lie algebras, local triviality of fibrations offer a suitable avenue for the bicocycle double cross product (local) Lie groups. Accordingly, we shall conclude this subsection with a brief discussion on the local Lie groups.

\subsubsection*{Local Lie groups}~

Let us first recall from \cite[Def. 2]{Olve96}, see also \cite[Sect. 3]{Ercu-book}, that a manifold $M$, with a distinguished element $e\in M$, is said to be a \emph{local Lie group} if there are smooth maps 
\[
\mu:U \to M, \qquad i:V \to M,
\]
where $U\subseteq M\times M$ is an open subset such that $(\{e\}\times M)\cup(M\times \{e\}) \subseteq U$, and $V\subseteq M$ is also an open subset with $e\in V$, so that $V\times i(V) \subseteq U$, and that $i(V)\times V \subseteq U$, satisfying
\begin{itemize}
\item[(i)] $\mu(e,x) = x = \mu(x,e)$, for any $x\in M$,
\item[(ii)] $\mu(i(y),y) = e = \mu(y,i(y))$, for any $y\in V$, and
\item[(iii)] $\mu(x,\mu(y,z)) = \mu(\mu(x,y),z)$ for any $(x,y), (y,z), (x,\mu(y,z)), (\mu(x,y),z) \in U$.
\end{itemize}

As was pointed out in \cite[Ex. 3]{Olve96}, an immediate example od a local Lie group is that of a neighbourhood of the identity in a (global) Lie group. To be more precise, let $G$ be a Lie group with the identity element $e\in G$, and $M\subseteq G$ to be any neighbourhood of the identity. Then, $M$ defines a local Lie group by restricting the group multiplication $m:G\times G \to G$ to any open subset $U \subseteq M\times M$ satisfying $(\{e\}\times M)\cup(M\times \{e\}) \subseteq U \subseteq (M\times M)\cap \mu^{-1}(M)$, and the inversion $i:G\to G$ to any open subset $V\subseteq M$ subject to $e \in V \subseteq M\cap i^{-1}(M)$ and $(V\times i(V)) \cup (i(V)\times V) \subseteq U$.

Accordingly, given a fibration
\[
\xymatrix{
F \ar[r]& G  \ar[r]^\pi& B
}
\]
so that the total space $G$ has the structure of a Lie group (such as the Hopf fibrations), the local triviality allows to decompose a neighbourhood $M\subseteq G$ of the identity $e\in G$, say into $M\cong M_1\times M_2$, along with $e \to (e_1,e_2)$. As such, any $U \subseteq M\times M$ decomposes as $U\cong U_1 \times U_2$, where $U_1 \subseteq M_1\times M_1$ and $U_2 \subseteq M_2\times M_2$. The conditions \eqref{phi-mu-map-gr-II}-\eqref{gamma-map-gr-assoc-III} of Proposition \ref{prop-bicocycle-double-cross-prod-gr} then allow to characterize (the associativity of) the multiplication 
\[
(x,h)(x',h') = \Big(x\cdot [\vp(h, x' ) \cdot\g(\psi(h,x'),h')],\,[\t(x,\vp(h, x'))\ast \psi(h ,x')] \ast h' \Big),
\]
via
\begin{align}
& \vp:p_1(U_2)\times p_2(U_1) \to p_1(m^{-1}(p_2(U_1))) \cap p_2(U_1), \qquad (h, x)\mapsto \vp(h,x), \label{left-action-gr-II-2} \\
& \psi:p_1(U_2)\times p_2(U_1) \to p_1(U_2) \cap p_2(m^{-1}(p_1(U_2))), \qquad (h, x)\mapsto \psi(h, x),\label{right-action-gr-II-2} \\
& \phi:U_1 \to M_1, \qquad (x, x')\mapsto \phi(x,x')=:x \cdot x', \label{phi-map-gr-II-2} \\
& \t:U_1 \to p_1(m^{-1}(p_1(U_2))), \qquad (x, x')\mapsto \t(x, x') , \label{theta-map-gr-II-2}, \\
& \mu:U_2 \to M_2, \qquad (h,h')\mapsto \mu(h,h')=:h\ast h' \label{mu-map-gr-II-2} \\
& \g:U_2 \to p_2(m^{-1}(p_2(U_1))), \qquad (h,h')\mapsto \g(h,h') \label{gamma-map-gr-II-2} 
\end{align}
satisfying
\begin{align}\label{normalization-gr-II-2}
\begin{split}
& \vp(e_2,x) = x, \qquad \vp(h,e_1) = e_1, \\
& \psi(h,e_1) = h, \qquad \psi(e_2,x) = e_2, \\
& e_1 \cdot e_1 = e_1, \qquad e_2\ast e_2 = e_2, \\
& \t(e_1,e_1) = e_2, \qquad \g(e_2,e_2) = e_1,
\end{split}
\end{align}
for any $(x,x')\in U_1$ and any $(h,h') \in U_2$, where given $i,j\in\{1,2\}$, $p_i(U_j) \in M_i$ refers to the projection of $U_j$ onto the $i$th component. Then, the decomposition $M\cong M_1\times M_2$ and Proposition \ref{prop-universal-II-gr} ensures the multiplication on $G$ coincides with the one above when restricted to $U \subseteq M\times M$. Lastly, in order to ensure \eqref{right-inverses-gr} and \eqref{left-inverses-gr}, we take a $V\cong V_1\times V_2 \in M \cap i^{-1}(M)$ such that $(V_1\times i(V_1)) \cup (i(V_1)\times V_1) \subseteq U_1$, and that $(V_2\times i(V_2)) \cup (i(V_2)\times V_2) \subseteq U_2$. 

Finally, we do note in this case that $M_1$ (resp. $M_2$) becomes a local Lie group via $U_1 \subseteq M_1 \times M_1$ (resp. $U_2 \subseteq M_2 \times M_2$) and $V_1 \subseteq M_1$ (resp. $V_2 \subseteq M_2$). As a result, we may say (with a slight abuse of language) that $M$ is a bicocycle double cross product of the local Lie groups $M_1$ and $M_2$.

\section{Quantum (Bi)cocycle Double Cross Constructions} \label{Sec-Quan}

Motivated by the fact that Hopf algebras are the quantum counterparts of both Lie groups and Lie algebras, we shall upgrade in this section the bicocycle double cross constructions of both Subsection \ref{subsect-bicocycle-double-sum} and Subsection \ref{subsect-Cocycle-double-cross-product-groups} to the level of bialgebras. Furthermore, we shall also discuss the semidualization of such a \emph{bicocycle double cross product bialgebra}, which happens to be a unified product as an algebra, and a \emph{unified coproduct} as a coalgebra. The latter will be constructed below in detail, under the name of \emph{cocycle double cross coproduct bialgebra}, while the former is precisely the construction in \cite{AgorMili11}, and hence will only be reviewed for the readers' convenience.

\subsection{Cocycle Double Cross Product Bialgebras}\label{subsect-Cocycle-double-cross-product-bialgebra}~

In the present subsection we shall recall from \cite{AgorMili11}, see also \cite{AgoreMilitaru-book}, a bialgebra construction which accounts for the quantum analogue of both the unified products $\G{g}:=\G{m}\bowtie_\t\G{h}$ of \cite{AgorMili14}, and $G\bowtie_\t M$ of \cite{AgorMili14-II}. Namely, we shall review the construction of a bialgebra $\C{G}:= \G{M} \ot \C{H}$, in which $\G{M}$ is merely a subcoalgebra, while $\C{H}$ is a sub-bialgebra. The construction may then be pictured as 
\[
\xymatrix{
\G{M} \ar@{^{(}->}[r]_i  & \C{G}    &  \ar@{_{(}->}[l]^j_{\rm alg}\C{H},
}
\]
in the category of coalgebras, and provides a natural cocycle generalization of the double cross product construction in \cite[Thm. 7.2.2]{Majid-book}, and the cocycle bicrossproduct construction in \cite[Prop. 6.3.7]{Majid-book}.  

Let $\C{H}$ be a bialgebra, and let $(\G{M},e)$ be a coalgebra equipped with a distinguished grouplike $e\in \G{M}$. Let the mutual interaction of the pair $(\G{M},\C{H})$ be given by
\begin{equation}\label{left-action-Hopf}
\vp:\C{H}\ot \G{M} \to \G{M}, \qquad h\ot x\mapsto \vp(h,x)=: h\rt x,
\end{equation}
that satisfies
\begin{equation}\label{vp-action-Hopf}
1\rt x = x,
\end{equation}
and
\begin{equation}\label{right-action-Hopf}
\psi:\C{H}\ot \G{M} \to \C{H}, \qquad h\ot x \mapsto \psi(h,x),
\end{equation}
that satisfies 
\begin{equation}\label{psi-unital-Hopf}
\psi(h,e) = h.
\end{equation} 

Let us emphasize that \eqref{left-action-Hopf} and \eqref{right-action-Hopf} are morphisms of coalgebras (as we work in the category of coalgebras); that is, 
\[
\D(h\rt x) = h\ps{1}\rt x\ps{1} \ot h\ps{2}\rt x\ps{2}, \qquad \ve(h\rt x) = \ve(h)\ve(x),
\]
as such, $\G{M}$ is a left $\C{H}$-module algebra, and 
\[
\D(\psi(h,x)) = \psi(h\ps{1},x\ps{1}) \ot \psi(h\ps{2},x\ps{2}), \qquad \ve(\psi(h,x)) = \ve(h)\ve(x).
\]
Furthermore, just like it is in \cite[Def. 7.2.1]{Majid-book}, we shall assume 
\begin{equation}\label{left-right-action-on-id-Hopf}
h\rt e=\ve(h)e, \qquad \psi(1,x) = \ve(x)1.
\end{equation}

Let us further assume that the coalgebra $\G{M}$ is endowed with a binary operation
\begin{equation}\label{phi-map-Hopf}
\phi:\G{M}\ot \G{M} \to \G{M}, \qquad x\ot x' \mapsto \phi(x,x')=:x\cdot x',
\end{equation}
as well as a mapping
\begin{equation}\label{theta-map-Hopf}
\t:\G{M}\ot \G{M} \to \C{H}, \qquad x\ot x' \mapsto \t(x,x'),
\end{equation}
which satisfies
\begin{equation}\label{theta-unital-Hopf}
\t(e,e) = 1
\end{equation}
in order to be able to realize $\C{H} \hookrightarrow{} \G{M}\ot \C{H}$ as a subalgebra. Let us underline also that these too are coalgebra morphisms; that is,
\[
\D(x\cdot x') = x\ps{1}\cdot x'\ps{1} \ot x\ps{2}\cdot x'\ps{2}, \qquad \ve(x\cdot x') = \ve(x)\ve(x'),
\]
and
\[
\D(\t(x, x')) = \t(x\ps{1}, x'\ps{1}) \ot \t(x\ps{2}, x'\ps{2}), \qquad \ve(\t(x, x')) = \ve(x)\ve(x').
\]

\begin{definition}
Let $\G{M}$ be a coalgebra with a distinguished group-like $e \in\G{M}$, and let $\C{H}$ be a bialgebra. Then the pair $(\G{M},\C{H})$, equipped with coalgebra maps \eqref{left-action-Hopf}, \eqref{right-action-Hopf}, \eqref{phi-map-Hopf}, and \eqref{theta-map-Hopf}, that satisfy \eqref{vp-action-Hopf}, \eqref{psi-unital-Hopf}, \eqref{left-right-action-on-id-Hopf}, \eqref{theta-unital-Hopf}  is called a (right) \emph{cocycle double cross product pair}. 
\end{definition}

We refer the reader to \cite[Def. 2.5]{BespDrab99} for a comparison with the \emph{Hopf datum}.

\begin{proposition}\label{prop-cocycle-double-cross-prod}
Let $(\G{M}, \C{H})$ be a (right) cocycle double cross product pair. Then, $\C{G}:=\G{M}\ot \C{H}$ is a bialgebra through
\begin{equation}\label{mult-cocycly-cross-prod-multp}
(x\ot h)(x'\ot h') := x\ps{1}\cdot (h\ps{1}\rt x'\ps{1}) \ot \t(x\ps{2},h\ps{2}\rt x'\ps{2})\psi(h\ps{3},x'\ps{3})h',
\end{equation}
and
\begin{equation}\label{comultp-cross-prod}
\D(x\ot h) := (x\ps{1}\ot h\ps{1}) \ot (x\ps{2}\ot h\ps{2}),
\end{equation}
with the unit $e\ot 1 \in \G{M}\ot \C{H}$, and the counit given by $\ve(x\ot h):= \ve(x)\ve(h)$, if and only if \eqref{left-action-Hopf}, \eqref{right-action-Hopf}, \eqref{phi-map-Hopf}, and \eqref{theta-map-Hopf} are subject to
\begin{align}
& e\cdot x = x = x\cdot e, \label{phi-map-Hopf-e} \\
& \t(x,e) = \ve(x)1 = \t(e,x), \label{theta-map-Hopf-e} \\
& h\rt (x\cdot x') = (h\ps{1}\rt x\ps{1})\cdot (\psi(h\ps{2}, x\ps{2}) \rt x'), \label{left-action-Hopf-on-multp} \\
& \psi(h,x\ps{1}\cdot x'\ps{1})\t(x\ps{2}, x'\ps{2}) = \t(h\ps{1}\rt x\ps{1}, \psi(h\ps{2}, x\ps{2}) \rt x'\ps{1})\psi(\psi(h\ps{3},x\ps{3}), x'\ps{2}), \label{right-action-theta-comp-Hopf} \\
& \psi(hh', x) = \psi(h, h'\ps{1}\rt x\ps{1})\psi(h'\ps{2}, x\ps{2}),  \label{right-action-Hopf-on-multp}  \\
& \G{M} \text{ is a left } \C{H}-\text{module}, \label{M-is-left-H-module} \\
& x\cdot (x'\cdot x'') = (x\ps{1}\cdot x'\ps{1}) \cdot (\t(x\ps{2}, x'\ps{2})\rt x''), \label{phi-map-Hopf-assoc} \\
& \t(x, x'\ps{1}\cdot x''\ps{1})\t(x'\ps{2}, x''\ps{2}) = \t(x\ps{1}\cdot x'\ps{1}, \t(x\ps{2}, x'\ps{2})\rt x''\ps{1})\psi(\t(x\ps{3}, x'\ps{3}), x''\ps{2}),  \label{theta-map-Hopf-cocycle} \\
& \psi(h\ps{2},x\ps{2}) \ot h\ps{1}\rt x\ps{1} = \psi(h\ps{1},x\ps{1}) \ot h\ps{2}\rt x\ps{2}, \label{left-right-action-comultp} \\
& \t(x\ps{2},x'\ps{2}) \ot x\ps{1}\cdot x'\ps{1} =  \t(x\ps{1},x'\ps{1}) \ot x\ps{2}\cdot x'\ps{2}. \label{theta-phi-comp}
\end{align}
\end{proposition}

The bialgebra $\G{M} \rtimes\C{H}:= \G{M}\ot\C{H}$ of Proposition \ref{prop-cocycle-double-cross-prod} is in fact the (right handed) unified product of $\G{M}$ and $\C{H}$, in \cite[Thm. 2.4]{AgorMili11}, and $\Om(\C{H}):=(\G{M},\rt,\psi,\t)$ an extending datum of the bialgebra $\C{H}$ in \cite[Def. 2.1]{AgorMili11}. We shall, on the other hand, use the notation $\G{M} \bowtie_\theta\C{H}$ to emphasize the twisted cocycle, and call it (right) \emph{cocycle double cross product}  in order to highlight its intermediate place in a hierarchy of constructions. The precise unified product in \cite[Thm. 2.4]{AgorMili11} is denoted by $\G{M} {\,}_{\g \hspace{-0.1cm}}\bowtie\C{H}$ in the present terminology, and may be referred as a (left) \emph{cocycle double cross product}.

A series of remarks are in order.

\begin{remark}
Let us note that the left action \eqref{left-action-Hopf} being a map of coalgebras is equivalent to $\G{M}$ being a (left) $\C{H}$-module coalgebra. Similarly, if in particular \eqref{right-action-Hopf} is a right action, that its being a coalgebra map will mean $\C{H}$ being a (right) $\G{M}$-module coalgebra. In this case, \eqref{left-action-Hopf-on-multp}, \eqref{right-action-Hopf-on-multp}, and \eqref{left-right-action-comultp} are nothing but (7.7), (7.8) and (7.9) of \cite[Def. 7.2.1]{Majid-book}. If, furthermore, \eqref{theta-map-Hopf} is trivial, then Proposition \ref{prop-cocycle-double-cross-prod} above coincides with  \cite[Thm. 7.2.2]{Majid-book}.
\end{remark}

\begin{remark}
If the left action \eqref{left-action-Hopf} is trivial (given by the counit), then the algebra $\G{M} \bowtie_\theta\C{H}$ given in Proposition \ref{prop-cocycle-double-cross-prod} is nothing but $\G{M} \ltimes_\theta\C{H}$ of \cite[Prop. 6.3.7]{Majid-book}, since in this case the condition (6.27) of \cite[Prop. 6.3.7]{Majid-book} follows from \eqref{right-action-Hopf-on-multp}, (6.28) from \eqref{right-action-theta-comp-Hopf}, and (6.29) from \eqref{theta-map-Hopf-cocycle}. If, furthermore, \eqref{right-action-Hopf} is a right action and the multiplication on $\C{H}$ is trivial; that is, given by the addition (in other words $\C{H}$ is regarded only as a vector space), then \eqref{theta-map-Hopf-cocycle} indicates that \eqref{theta-map-Hopf} is a 2-cocycle in the algebra Hochschild cohomology of $\G{M}$ with coefficients in $\C{H}$. In short,
\[
\t\in H^2(\G{M},\C{H}).
\]
\end{remark}

\begin{remark}
Finally, let us mention that the cocycle cross product construction does fit into the more general Brzezinski cross product construction (cross product with a vector space), \cite[Prop. 2.1]{Brze97-II}. More precisely, it falls into the realm of  \cite[Prop. 2.2]{Brze97-II}; cross product by a coalgebra construction, associated to the entwining data $(\C{H},\G{M},\rho,e,\rho^\G{M})$ given by
\[
\rho:\C{H}\ot \G{M} \to \G{M}\ot \C{H}, \qquad \rho(h\ot x) := h\ps{1}\rt x \ot h\ps{2},
\]
and
\[
\rho^\G{M}:\G{M}\ot \G{M} \to \G{M}\ot \G{M}, \qquad \rho^\G{M}(x\ot x') := x\ps{1} \cdot x'\ot x\ps{2}.
\]
\end{remark}

An analogue of \cite[Thm. 7.2.3]{Majid-book} and \cite[Thm. 2.7]{AgorMili11}, compare with \cite[Prop. 2.3]{BespDrab01}, is stated below. A more general statement will be proved below, and hence the proof is omitted.  

\begin{proposition}\label{prop-universal-I}
Given a coalgebra $\G{M}$, and two bialgebras $\C{G}$ and $\C{H}$ that fit into
\[
\xymatrix{
\G{M} \ar@{^{(}->}[r]_i  & \C{G}    &  \ar@{_{(}->}[l]^j_{\rm alg}\C{H},
}
\]
in the category of coalgebras, if $\mu\circ (i\ot j):\G{M}\ot \C{H} \to\C{G}$ is an isomorphism (of coalgebras), where $\mu:\C{G}\ot \C{G}\to \C{G}$ denotes the multiplication in $\C{G}$, then $\C{G}\cong \G{M}\bowtie_\t \C{H}$ as bialgebras. In this case, the maps \eqref{left-action-Hopf}, \eqref{right-action-Hopf}, \eqref{phi-map-Hopf}, \eqref{theta-map-Hopf}  are obtained by
\begin{equation}\label{maps-cocyc-double-cross-prod-bialg}
hx = (h\ps{1}\rt  x\ps{1}) \psi(h\ps{2},x\ps{2}), \qquad xx' = \phi(x\ps{1}, x'\ps{1})\t(x\ps{2},x'\ps{2}),
\end{equation}
for any $x,x' \in \G{M}$, and any $h \in \C{H}$.
\end{proposition}

Using now Proposition \ref{prop-universal-I}, we can express the universal enveloping algebra $U(\G{g})$ of a cocycle double cross sum Lie algebra $\G{g}:=\G{m}\bowtie_\t\G{h}$ as a cocycle double cross product algebra.

\begin{corollary}\label{prop-U-cocycle-double-cross-sum}
Given the cocycle double cross sum Lie algebra $\G{g}:=\G{m}\bowtie_\t\G{h}$, let $\widetilde{U}(\G{m}):= U(\G{g}) / U(\G{g})U(\G{h})^+$, where $U(\G{h})^+ := \ker\ve|_{U(\G{h})}$ for $\ve:U(\G{g})\to k$.
Then, $U(\G{m}\bowtie_\t\G{h}) \cong \widetilde{U}(\G{m}) \bowtie_\t U(\G{h})$ as bialgebras.
\end{corollary}

\begin{proof}
A vector space basis for $\widetilde{U}(\G{m})$ may be given by
\[
\{\xi_1^{r_1}\ldots \xi_m^{r_m}\mid r_1,\ldots, r_m \geq 0\},
\]
where $\{\xi_1,\ldots,\xi_m\} \subseteq \G{g}$ is a basis for the complement of $\G{h}\subseteq \G{g}$, see also \cite[Prop. 2.2.7]{Dixmier-book}. Similarly, a vector space basis for $U(\G{h})$ may be given by
\[
\{\xi_{m+1}^{r_{m+1}}\ldots \xi_n^{r_n} \mid r_{m+1},\ldots, r_n \geq 0\},
\]
where $\{\xi_{m+1},\ldots,\xi_n\} \subseteq \G{g}$ is a basis for $\G{h}$. On the other hand, both $\widetilde{U}(\G{m})$ and $U(\G{h})$ are subcoalgebras of $U(\G{g})$, with respect to the trivial (vector space basis elements of $\G{g}$ being primitive) vector space structure. Accordingly, we have 
\[
\xymatrix{
\widetilde{U}(\G{m}) \ar@{^{(}->}[r]_i  & U(\G{g})    &  \ar@{_{(}->}[l]^j_{\rm alg}  U(\G{h}).
}
\]
The claim then follows from Proposition \ref{prop-universal-I}. 
\end{proof}

\begin{remark}
The cocycle for the cocycle double cross product bialgebra is determined as in \eqref{maps-cocyc-double-cross-prod-bialg}, though with a slight abuse of notation we used the same symbol as the one in the Lie algebra level.
\end{remark}

\subsection{Cocycle Double Cross Coproduct Bialgebras}
\label{subsect-Cocycle-double-cross-coproduct-bialgebra}~

In the present subsection we shall present a (categorical) dual construction to that of the above subsection. More precisely, given a bialgebra $\C{H}$, and an algebra $\G{M}$ with a distinguished character $\eta:\G{M}\to k$, we shall consider a bialgebra structure on $\G{M} \dcc^\s \C{H} := \G{M}\ot \C{H}$, that fits into the picture
\[
\xymatrix{
\G{M}  &\ar@{->>}[l]_q \C{L}  \ar@{->>}[r]^p_{\rm coalg} &  \C{H}  
}
\]
in the category of algebras (and algebra homomorphisms).

To this end, let 
\begin{equation}\label{left-coaction-Hopf-new}
\nb:\G{M}\to \C{H}\ot \G{M}, \qquad x\mapsto \nb(x):=x\ns{-1}\ot x\ns{0},
\end{equation}
satisfy
\begin{equation}\label{nb-map-coaction-Hopf}
\ve(x\ns{-1})x\ns{0} = x,
\end{equation}
for any $x\in \G{M}$, and let there also be a linear map (not a priori coaction)
\begin{equation}\label{right-coaction-Hopf-new}
\Db:\C{H}\to \C{H}\ot \G{M}, \qquad h\mapsto \Db(h):=h\nsb{0}\ot h\nsb{1}
\end{equation}
that satisfies
\begin{equation}\label{Db-map-Hopf-counital}
 h\nsb{0}\eta(h\nsb{1}) = h
\end{equation}
for any $h\in \C{H}$. Analogue to \eqref{left-right-action-on-id-Hopf}; we shall assume that \eqref{left-coaction-Hopf-new} and \eqref{right-coaction-Hopf-new} interacts with the counits via
\begin{equation}\label{left-right-coaction-on-character-Hopf-new}
x\ns{-1}\eta(x\ns{0}) = \eta(x)1, \qquad  \ve(h\nsb{0})h\nsb{1}  = \ve(h)1.
\end{equation}
Furthermore, let there be the linear maps
\begin{equation}\label{delta-map-Hopf}
\d:\G{M}\to \G{M}\ot \G{M}, \qquad x\mapsto \d(x):=x\pr{1}\ot x\pr{2}
\end{equation}
and
\begin{equation}\label{sigma-map-Hopf}
\s:\C{H}\to \G{M}\ot \G{M}, \qquad h\mapsto \s(h):= h^{\pr{1}}\ot h^{\pr{2}}
\end{equation}
the latter of which satisfying
\begin{equation}\label{sigma-map-Hopf-counital}
\eta (h^{\pr{1}})\eta (h^{\pr{2}} )= \ve(h),
\end{equation}
for any $h\in \C{H}$, to be able to realize the projection $\G{M}\ot \C{H}\to \C{H}$ as a coalgebra homomorphism. 

Just as we work in the category of coalgebras in Subsection \ref{subsect-Cocycle-double-cross-product-bialgebra}, this time we shall work in the category of algebras; that is, we shall assume that the maps \eqref{left-coaction-Hopf-new}, \eqref{right-coaction-Hopf-new}, \eqref{delta-map-Hopf}, and \eqref{sigma-map-Hopf} are algebra homomorphisms (assuming the tensor product algebra structure on $\G{M}\ot \C{H}$). More explicitly,
\begin{align}\label{multiplicativity-of-maps}
\begin{split}
& \nb(xx') = x\ns{-1}x'\ns{-1}\ot x\ns{0}x'\ns{0}, \qquad \nb(1) = 1 \ot 1, \\
& \Db(hh') = h\nsb{0}h'\nsb{0}\ot h\nsb{1}h'\nsb{1}, \qquad \Db(1) = 1 \ot 1, \\
& \d(xx') =  x\pr{1}x'\pr{1}\ot x\pr{2}x'\pr{2}, \qquad \d(1) = 1\ot 1, \\
& \s(hh') =  h^{\pr{1}}{h'}^{\pr{1}}\ot h^{\pr{2}}{h'}^{\pr{2}}, \qquad \s(1) = 1\ot 1,
\end{split}
\end{align}
for all $x,x' \in \G{M}$, and any $h,h'\in \C{H}$.

\begin{definition}
Let $\G{M}$ be an algebra with a distinguished character $\eta:\G{M}\to k$, and let $\C{H}$ be a bialgebra. Then the pair $(\G{M},\C{H})$, equipped with algebra maps \eqref{left-coaction-Hopf-new}, \eqref{right-coaction-Hopf-new}, \eqref{delta-map-Hopf}, and \eqref{sigma-map-Hopf}, that satisfy \eqref{nb-map-coaction-Hopf}, \eqref{Db-map-Hopf-counital}, \eqref{left-right-coaction-on-character-Hopf-new}, \eqref{sigma-map-Hopf-counital} is called a ``(right) cocycle double cross coproduct pair.'' 
\end{definition}

Once again, we refer the reader to \cite[Def. 2.5]{BespDrab99} for a comparison with the ``Hopf datum''. Now, analogue to Proposition \ref{prop-cocycle-double-cross-prod}, we have the following.

\begin{proposition}\label{prop-cocycle-double-cross-coprod}
Let $(\G{M}, \C{H})$ be a (right) cocycle double cross coproduct pair. Then, the tensor product $\C{G}:=\C{H}\ot \G{M}$ is a bialgebra through
\begin{equation}\label{multp-cross-coprod}
(x\ot h)(x'\ot h') := xx' \ot hh',
\end{equation}
and
\begin{equation}\label{comult-cocycly-cross-coprod-comultp}
\D_{\dcc}(x\ot h) := (x\pr{1}{h\ps{1}}^{\pr{1}}\ot x\pr{2}\ns{-1}{h\ps{1}}^{\pr{2}}\ns{-1} h\ps{2}\nsb{0} )\ot (x\pr{2}\ns{0}{h\ps{1}}^{\pr{2}}\ns{0}h\ps{2}\nsb{1}\ot h\ps{3}),
\end{equation}
with the unit $1\ot 1 \in \G{M}\ot \C{H}$ and the counit $\ve_{\dcc}(x\ot h):= \eta(x)\ve(h)$, if and only if \eqref{left-coaction-Hopf-new}, \eqref{right-coaction-Hopf-new}, \eqref{delta-map-Hopf}, and \eqref{sigma-map-Hopf} are subject to
\begin{align}
& \eta(x\pr{1})x\pr{2} = x = x\pr{1}\eta(x\pr{2}), \label{delta-counital} \\
& \eta (h^{\pr{1}})h^{\pr{2}} = \ve(h)1 = h^{\pr{1}} \eta(h^{\pr{2}}), \label{sigma-counital} \\
& \G{M} \text{ is a left } \C{H}-{comodule}, \label{M-is-left-H-comodule} \\
& x\pr{1} \ot x\pr{2}\pr{1} \ot x\pr{2}\pr{2} = x\pr{1}\pr{1}{x\pr{2}\ns{-1}}^{\pr{1}} \ot x\pr{1}\pr{2}{x\pr{2}\ns{-1}}^{\pr{2}} \ot x\pr{2}\ns{0}, \label{delta-coassoc} \\
& x\ns{-1} \ot x\ns{0}\pr{1}\ot  x\ns{0}\pr{2} = x\pr{1}\ns{-1} x\pr{2}\ns{-1}\nsb{0} \ot x\pr{1}\ns{0}  x\pr{2}\ns{-1}\nsb{1} \ot x\pr{2}\ns{0}, \label{left-comod-coalg} \\
& {h\ps{1}}^{\pr{1}} \ot {h\ps{1}}^{\pr{2}}\pr{1}{h\ps{2}}^{\pr{1}} \ot {h\ps{1}}^{\pr{2}}\pr{2}{h\ps{2}}^{\pr{2}} = \notag\\
& \hspace{1.5cm} {h\ps{1}}^{\pr{1}}\pr{1} {{h\ps{1}}^{\pr{2}}\ns{-1}}^{\pr{1}} {h\ps{2}\nsb{0}}^{\pr{1}} \ot  {h\ps{1}}^{\pr{1}}\pr{2} {{h\ps{1}}^{\pr{2}}\ns{-1}}^{\pr{2}} {h\ps{2}\nsb{0}}^{\pr{2}} \ot {h\ps{1}}^{\pr{2}}\ns{0} h\ps{2}\nsb{1},  \label{h-h-h-I} \\
& h\ps{1}\nsb{0} \ot  h\ps{1}\nsb{1}\pr{1} {h\ps{2}}^{\pr{1}} \ot h\ps{1}\nsb{1}\pr{2} {h\ps{2}}^{\pr{2}} = \notag \\
& \hspace{3cm}  {h\ps{1}}^{\pr{1}}\ns{-1} {h\ps{1}}^{\pr{2}}\ns{-1}\nsb{0} h\ps{2}\nsb{0}\nsb{0} \ot {h\ps{1}}^{\pr{1}}\ns{0} {h\ps{1}}^{\pr{2}}\ns{-1}\nsb{1} h\ps{2}\nsb{0}\nsb{1} \ot {h\ps{1}}^{\pr{2}}\ns{0} h\ps{2}\nsb{1}, \label{h-h-h-II} \\
& h\ps{1}\nsb{0} \ot h\ps{1}\nsb{1}\ns{-1} h\ps{2}\nsb{0} \ot h\ps{1}\nsb{1}\ns{0} h\ps{2}\nsb{1} = h\nsb{0}\ps{1} \ot h\nsb{0}\ps{2} \ot h\nsb{1}, \label{h-h-h-III} \\
& x\ns{-1}h\nsb{0}\ot x\ns{0}h\nsb{1} = h\nsb{0}x\ns{-1}\ot h\nsb{1}x\ns{0}, \label{x-h-switch} \\
& h^{\pr{1}}x\pr{1} \ot h^{\pr{2}}x\pr{2} = x\pr{1}h^{\pr{1}} \ot x\pr{2}h^{\pr{2}} \label{sigma-delta-switch}
\end{align}
for any $x \in \G{M}$, and any $h \in\C{H}$.
\end{proposition}

\begin{proof}
Let us begin with the counitality of \eqref{comult-cocycly-cross-coprod-comultp}. To this end, we first note that
\begin{align*}
& (\ve_{\dcc}\ot \Id_{\G{M}\ot \C{H}}) \circ \D_{\dcc}(x\ot h) = \\
& \eta(x\pr{1})\eta({h\ps{1}}^{\pr{1}})\ve(x\pr{2}\ns{-1})\ve({h\ps{1}}^{\pr{2}}\ns{-1})\ve(h\ps{2}\nsb{0}) x\pr{2}\ns{0}{h\ps{1}}^{\pr{2}}\ns{0}h\ps{2}\nsb{1}\ot h\ps{3} = \\
& \eta(x\pr{1})\eta({h\ps{1}}^{\pr{1}}) x\pr{2}{h\ps{1}}^{\pr{2}}\ot h\ps{2}.
\end{align*}
Accordingly,
\begin{equation}\label{dcc-left-counital}
(\ve_{\dcc}\ot \Id_{\G{M}\ot \C{H}}) \circ \D_{\dcc}(x\ot h) = x\ot h
\end{equation}
implies, applying $\Id_{\G{M}} \ot \ve: \G{M}\ot \C{H} \to \G{M}$, 
\[
\eta(x\pr{1})\eta(h^{\pr{1}})x\pr{2}h^{\pr{2}} = \ve(h)x
\]
for any $x\in \G{M}$, and any $h\in \C{H}$. Hence, in particular, for $h=1$ we get
\begin{equation}\label{delta-counital-left}
\eta(x\pr{1})x\pr{2} = x,
\end{equation}
and for $x=1$ we obtain
\begin{equation}\label{sigma-counital-left}
\eta(h^{\pr{1}})h^{\pr{2}} = \ve(h)1.
\end{equation}
Conversely, \eqref{delta-counital-left} and \eqref{sigma-counital-left} together implies \eqref{dcc-left-counital}. On the other hand, we have
\begin{align*}
& (\Id_{\G{M}\ot \C{H}} \ot \ve_{\dcc}) \circ \D_{\dcc}(x\ot h) = \\
& x\pr{1}{h\ps{1}}^{\pr{1}}\ot x\pr{2}\ns{-1}{h\ps{1}}^{\pr{2}}\ns{-1} h\ps{2}\nsb{0} \eta(x\pr{2}\ns{0})\eta({h\ps{1}}^{\pr{2}}\ns{0})\eta(h\ps{2}\nsb{1})\ve(h\ps{3}) = \\
& x\pr{1}{h\ps{1}}^{\pr{1}}\ot  h\ps{2} \eta(x\pr{2})\eta({h\ps{1}}^{\pr{2}}).
\end{align*}
As a result, 
\begin{equation}\label{dcc-counital-right}
(\Id_{\G{M}\ot \C{H}} \ot \ve_{\dcc}) \circ \D_{\dcc}(x\ot h) = x\ot h
\end{equation}
implies, applying once again $\Id_{\G{M}} \ot \ve: \G{M}\ot \C{H} \to \G{M}$,
\[
x\pr{1}h^{\pr{1}} \eta(x\pr{2})\eta(h^{\pr{2}}) = \ve(h)x.
\]
This time, $h=1$ yields
\begin{equation}\label{delta-counital-right}
x\pr{1}\eta(x\pr{2}) = x,
\end{equation}
while $x=1$ gives
\begin{equation}\label{sigma-counital-right}
h^{\pr{1}}\eta(h^{\pr{2}}) = \ve(h)1.
\end{equation}
Moreover, similar to the arguments above; \eqref{delta-counital-right} and \eqref{sigma-counital-right} together implies \eqref{dcc-counital-right}. We may thus conclude that \eqref{comult-cocycly-cross-coprod-comultp} is counital if and only if \eqref{delta-counital} and \eqref{sigma-counital} hold.

Let us, on the next step, show that the multiplication \eqref{multp-cross-coprod} is comultiplicative; 
\begin{equation}\label{multp-is-comultp}
A\ps{1}A'\ps{1} \ot A\ps{2}A'\ps{2} = (AA')\ps{1} \ot (AA')\ps{2}
\end{equation}
for any $A,A' \in \G{M}\ot \C{H}$. To this end, it suffices to consider only $A:= 1\ot h\in \G{M}\ot \C{H}$, and $A':=x\ot 1\in \G{M}\ot \C{H}$. The identity \eqref{multp-is-comultp} thus becomes
\begin{align*}
& ({h\ps{1}}^{\pr{1}}x\pr{1}\ot {h\ps{1}}^{\pr{2}}\ns{-1} h\ps{2}\nsb{0} x\pr{2}\ns{-1})\ot ({h\ps{1}}^{\pr{2}}\ns{0}h\ps{2}\nsb{1}x\pr{2}\ns{0}\ot h\ps{3}) = \\
& (x\pr{1}{h\ps{1}}^{\pr{1}}\ot x\pr{2}\ns{-1}{h\ps{1}}^{\pr{2}}\ns{-1} h\ps{2}\nsb{0} )\ot (x\pr{2}\ns{0}{h\ps{1}}^{\pr{2}}\ns{0}h\ps{2}\nsb{1}\ot h\ps{3}).
\end{align*}
Then, the application of
\[
\eta\ot \Id_{\C{H}} \ot \Id_{\G{M}} \ot \ve: \G{M}\ot \C{H} \ot \G{M}\ot \C{H} \to \C{H}\ot \G{M}
\]
yields \eqref{x-h-switch}, while the application of 
\[
\Id_{\G{M}}\ot \ve  \ot \Id_{\G{M}\ot \C{H}} \ot \ve: \G{M}\ot \C{H} \ot \G{M}\ot \C{H} \to \G{M}\ot \G{M}
\]
gives \eqref{sigma-delta-switch}. Conversely, \eqref{x-h-switch} and \eqref{sigma-delta-switch} together imply \eqref{multp-is-comultp}.

As for the coassociativity, it suffices to check only for $x\ot 1\in \G{M}\ot \C{H}$ and $1\ot h\in \G{M}\ot \C{H}$, as a result of the multiplicativity of \eqref{comult-cocycly-cross-coprod-comultp}. 

As for the elements of the form $x\ot 1\in \G{M}\ot \C{H}$, we have
\begin{align*}
& (\Id_{\G{M}\ot \C{H}} \ot \D_{\dcc})\circ \D_{\dcc} (x\ot 1)  = \\
& (x\pr{1}\ot x\pr{2}\ns{-1})\ot \D_{\dcc}(x\pr{2}\ns{0}\ot 1) = \\
& \Big(x\pr{1}\ot x\pr{2}\ns{-1}\Big)\ot \Big(x\pr{2}\ns{0}\pr{1}\ot x\pr{2}\ns{0}\pr{2}\ns{-1}\Big)\ot \Big(x\pr{2}\ns{0}\pr{2}\ns{0}\ot 1\Big)
\end{align*}
and
\begin{align*}
& (\D_{\dcc} \ot \Id_{\G{M}\ot \C{H}} )\circ \D_{\dcc} (x\ot 1)  = \\
& \D_{\dcc}(x\pr{1}\ot x\pr{2}\ns{-1})\ot (x\pr{2}\ns{0}\ot 1) = \\
& \Big(x\pr{1}\pr{1}{x\pr{2}\ns{-1}\ps{1}}^{\pr{1}}\ot x\pr{1}\pr{2}\ns{-1}{x\pr{2}\ns{-1}\ps{1}}^{\pr{2}}\ns{-1} x\pr{2}\ns{-1}\ps{2}\nsb{0} \Big)\ot \\
& \hspace{3cm} \Big(x\pr{1}\pr{2}\ns{0}{x\pr{2}\ns{-1}\ps{1}}^{\pr{2}}\ns{0}x\pr{2}\ns{-1}\ps{2}\nsb{1}\ot x\pr{2}\ns{-1}\ps{3}\Big) \ot \Big(x\pr{2}\ns{0}\ot 1\Big).
\end{align*}
Accordingly, applying 
\begin{equation}\label{eta-eta-ve}
\eta \ot \Id_\C{H} \ot\eta \ot \Id_\C{H} \ot\Id_\G{M} \ot \ve :\G{M}\ot \C{H}\ot \G{M}\ot \C{H} \ot \G{M}\ot \C{H} \to \C{H}\ot \C{H}\ot \G{M},
\end{equation}
to
\begin{equation}\label{comultp-x}
(\Id_{\G{M}\ot \C{H}} \ot \D_{\dcc})\circ \D_{\dcc} (x\ot 1)  = (\D_{\dcc} \ot \Id_{\G{M}\ot \C{H}} )\circ \D_{\dcc} (x\ot 1)
\end{equation}
we obtain \eqref{M-is-left-H-comodule}. Similarly, the application of 
\begin{equation}\label{ve-ve-ve}
\Id_\G{M} \ot \ve \ot\Id_\G{M} \ot \ve \ot\Id_\G{M} \ot \ve :\G{M}\ot \C{H}\ot \G{M}\ot \C{H} \ot \G{M}\ot \C{H} \to \G{M}\ot \G{M}\ot \G{M},
\end{equation}
yields \eqref{delta-coassoc}, while the application of
\begin{equation}\label{eta-ve-ve}
\eta \ot \Id_\C{H} \ot \Id_\G{M}\ot \ve \ot \Id_\G{M} \ot \ve :\G{M}\ot \C{H}\ot \G{M}\ot \C{H} \ot \G{M}\ot \C{H} \to \C{H}\ot \G{M}\ot \G{M}
\end{equation}
yields \eqref{left-comod-coalg}. Conversely, assuming \eqref{M-is-left-H-comodule}, \eqref{delta-coassoc}, and \eqref{left-comod-coalg}, we may observe that
\begin{align*}
& (\Id_{\G{M}\ot \C{H}} \ot \D_{\dcc})\circ \D_{\dcc} (x\ot 1)  = \\
& \Big(x\pr{1}\ot x\pr{2}\ns{-1}\Big)\ot \Big(x\pr{2}\ns{0}\pr{1}\ot x\pr{2}\ns{0}\pr{2}\ns{-1}\Big)\ot \Big(x\pr{2}\ns{0}\pr{2}\ns{0}\ot 1\Big) = \\
& \Big(x\pr{1}\ot x\pr{2}\pr{1}\ns{-1}x\pr{2}\pr{2}\ns{-1}\nsb{0}\Big)\ot \Big(x\pr{2}\pr{1}\ns{0}x\pr{2}\pr{2}\ns{-1}\nsb{1}\ot x\pr{2}\pr{2}\ns{0}\ns{-1}\Big)\ot \Big(x\pr{2}\pr{2}\ns{0}\ns{0}\ot 1\Big) = \\
& \Big(x\pr{1}\ot x\pr{2}\pr{1}\ns{-1}x\pr{2}\pr{2}\ns{-2}\nsb{0}\Big)\ot \Big(x\pr{2}\pr{1}\ns{0}x\pr{2}\pr{2}\ns{-2}\nsb{1}\ot x\pr{2}\pr{2}\ns{-1}\Big)\ot \Big(x\pr{2}\pr{2}\ns{0}\ot 1\Big) = \\
& \Big(x\pr{1}\pr{1}{x\pr{2}\ns{-1}}^{\pr{1}}\ot x\pr{1}\pr{2}\ns{-1}{x\pr{2}\ns{-1}}^{\pr{2}}\ns{-1}x\pr{2}\ns{0}\ns{-2}\nsb{0}\Big)\ot \\
& \hspace{3cm}\Big(x\pr{1}\pr{2}\ns{0} {x\pr{2}\ns{-1}}^{\pr{2}}\ns{0}x\pr{2}\ns{0}\ns{-2}\nsb{1}\ot x\pr{2}\ns{0}\ns{-1}\Big)\ot \Big(x\pr{2}\ns{0}\ns{0}\ot 1\Big) = \\
& \Big(x\pr{1}\pr{1}{x\pr{2}\ns{-3}}^{\pr{1}}\ot x\pr{1}\pr{2}\ns{-1}{x\pr{2}\ns{-3}}^{\pr{2}}\ns{-1} x\pr{2}\ns{-2}\nsb{0} \Big)\ot \\
& \hspace{3cm} \Big(x\pr{1}\pr{2}\ns{0}{x\pr{2}\ns{-3}}^{\pr{2}}\ns{0}x\pr{2}\ns{-2}\nsb{1}\ot x\pr{2}\ns{-1}\Big) \ot \Big(x\pr{2}\ns{0}\ot 1\Big) = \\
& (\D_{\dcc} \ot \Id_{\G{M}\ot \C{H}} )\circ \D_{\dcc} (x\ot 1),
\end{align*}
where in the second equality we used \eqref{left-comod-coalg}, and in the fourth equality we used \eqref{delta-coassoc}.

Now, for the elements of the form $1\ot h\in \G{M}\ot \C{H}$, on the one hand we have
\begin{align*}
& (\Id_{\G{M}\ot \C{H}} \ot \D_{\dcc})\circ \D_{\dcc} (1\ot h)  = \\
& ({h\ps{1}}^{\pr{1}}\ot {h\ps{1}}^{\pr{2}}\ns{-1} h\ps{2}\nsb{0} )\ot \D_{\dcc}({h\ps{1}}^{\pr{2}}\ns{0}h\ps{2}\nsb{1}\ot h\ps{3}) = \\
& \Big({h\ps{1}}^{\pr{1}}\ot {h\ps{1}}^{\pr{2}}\ns{-1} h\ps{2}\nsb{0} \Big)\ot \\
& \Big({h\ps{1}}^{\pr{2}}\ns{0}\pr{1}h\ps{2}\nsb{1}\pr{1}{h\ps{3}}^{\pr{1}}\ot {h\ps{1}}^{\pr{2}}\ns{0}\pr{2}\ns{-1}h\ps{2}\nsb{1}\pr{2}\ns{-1}{h\ps{3}}^{\pr{2}}\ns{-1}h\ps{4}\nsb{0}\Big) \ot \\
& \Big({h\ps{1}}^{\pr{2}}\ns{0}\pr{2}\ns{0} h\ps{2}\nsb{1}\pr{2}\ns{0}{h\ps{3}}^{\pr{2}}\ns{0}h\ps{4}\nsb{1}\ot  h\ps{5}\Big),
\end{align*}
while on the other hand 
\begin{align*}
& (\D_{\dcc}\ot \Id_{\G{M}\ot \C{H}} )\circ \D_{\dcc} (1\ot h) = \\
& \D_{\dcc}({h\ps{1}}^{\pr{1}}\ot {h\ps{1}}^{\pr{2}}\ns{-1} h\ps{2}\nsb{0} )\ot ({h\ps{1}}^{\pr{2}}\ns{0}h\ps{2}\nsb{1}\ot h\ps{3}) = \\
& \Big({h\ps{1}}^{\pr{1}}\pr{1}{({h\ps{1}}^{\pr{2}}\ns{-1} h\ps{2}\nsb{0})\ps{1}}^{\pr{1}}\ot   {h\ps{1}}^{\pr{1}}\pr{2}\ns{-1}{({h\ps{1}}^{\pr{2}}\ns{-1} h\ps{2}\nsb{0})\ps{1}}^{\pr{2}}\ns{-1} ({h\ps{1}}^{\pr{2}}\ns{-1} h\ps{2}\nsb{0})\ps{2}\nsb{0} \Big) \ot \\
& \Big({h\ps{1}}^{\pr{1}}\pr{2}\ns{0}{({h\ps{1}}^{\pr{2}}\ns{-1} h\ps{2}\nsb{0})\ps{1}}^{\pr{2}}\ns{0}({h\ps{1}}^{\pr{2}}\ns{-1} h\ps{2}\nsb{0})\ps{2}\nsb{1}\ot ({h\ps{1}}^{\pr{2}}\ns{-1} h\ps{2}\nsb{0})\ps{3}\Big) \ot \\
& \Big({h\ps{1}}^{\pr{2}}\ns{0}h\ps{2}\nsb{1}\ot h\ps{3}\Big),
\end{align*}
where, in view of \eqref{multiplicativity-of-maps}, 
\begin{align*}
& (\D_{\dcc}\ot \Id_{\G{M}\ot \C{H}} )\circ \D_{\dcc} (1\ot h) = \\
& \Big({h\ps{1}}^{\pr{1}}\pr{1} {{h\ps{1}}^{\pr{2}}\ns{-1}\ps{1}}^{\pr{1}} {h\ps{2}\nsb{0}\ps{1}}^{\pr{1}}\ot  {h\ps{1}}^{\pr{1}}\pr{2}\ns{-1}  {{h\ps{1}}^{\pr{2}}\ns{-1}\ps{1}}^{\pr{2}}\ns{-1} {h\ps{2}\nsb{0}\ps{1}}^{\pr{2}}\ns{-1}   {h\ps{1}}^{\pr{2}}\ns{-1}\ps{2}\nsb{0} h\ps{2}\nsb{0}\ps{2}\nsb{0} \Big) \ot \\
& \Big({h\ps{1}}^{\pr{1}}\pr{2}\ns{0}  {{h\ps{1}}^{\pr{2}}\ns{-1}\ps{1}}^{\pr{2}}\ns{0} {h\ps{2}\nsb{0}\ps{1}}^{\pr{2}}\ns{0}  {h\ps{1}}^{\pr{2}}\ns{-1}\ps{2}\nsb{1} h\ps{2}\nsb{0}\ps{2}\nsb{1} \ot    {h\ps{1}}^{\pr{2}}\ns{-1}\ps{3} h\ps{2}\nsb{0}\ps{3}\Big) \ot \\
&  \Big({h\ps{1}}^{\pr{2}}\ns{0}h\ps{2}\nsb{1}\ot h\ps{3}\Big) = \\
& \Big({h\ps{1}}^{\pr{1}}\pr{1} {{h\ps{1}}^{\pr{2}}\ns{-3}}^{\pr{1}} {h\ps{2}\nsb{0}\ps{1}}^{\pr{1}}\ot  {h\ps{1}}^{\pr{1}}\pr{2}\ns{-1}  {{h\ps{1}}^{\pr{2}}\ns{-3}}^{\pr{2}}\ns{-1} {h\ps{2}\nsb{0}\ps{1}}^{\pr{2}}\ns{-1}   {h\ps{1}}^{\pr{2}}\ns{-2}\nsb{0} h\ps{2}\nsb{0}\ps{2}\nsb{0} \Big) \ot \\
& \Big({h\ps{1}}^{\pr{1}}\pr{2}\ns{0}  {{h\ps{1}}^{\pr{2}}\ns{-3}}^{\pr{2}}\ns{0} {h\ps{2}\nsb{0}\ps{1}}^{\pr{2}}\ns{0}  {h\ps{1}}^{\pr{2}}\ns{-2}\nsb{1} h\ps{2}\nsb{0}\ps{2}\nsb{1} \ot    {h\ps{1}}^{\pr{2}}\ns{-1} h\ps{2}\nsb{0}\ps{3}\Big) \ot  \Big({h\ps{1}}^{\pr{2}}\ns{0}h\ps{2}\nsb{1}\ot h\ps{3}\Big).
\end{align*}
Accordingly, the application of \eqref{ve-ve-ve} to
\begin{equation}\label{comultp-h}
(\Id_{\G{M}\ot \C{H}} \ot \D_{\dcc})\circ \D_{\dcc} (1\ot h)  = (\D_{\dcc} \ot \Id_{\G{M}\ot \C{H}} )\circ \D_{\dcc} (1\ot h)
\end{equation}
yields \eqref{h-h-h-I}, while the application of \eqref{eta-ve-ve} produces \eqref{h-h-h-II}, and finally
\[
\eta \ot \Id_\C{H} \ot \eta \ot \Id_\C{H} \ot \Id_\G{M} \ot \ve :\G{M}\ot \C{H}\ot \G{M}\ot \C{H} \ot \G{M}\ot \C{H} \to \C{H}\ot \C{H}\ot \G{M}
\]
gives \eqref{h-h-h-III}. Conversely, \eqref{h-h-h-I}, \eqref{h-h-h-II}, and \eqref{h-h-h-III} together imply the coassociativity on $1\ot h \in \G{M}\ot \C{H}$. Indeed, using \eqref{h-h-h-III},
\begin{align*}
& (\D_{\dcc}\ot \Id_{\G{M}\ot \C{H}} )\circ \D_{\dcc} (1\ot h) = \\
& \Big({h\ps{1}}^{\pr{1}}\pr{1} {{h\ps{1}}^{\pr{2}}\ns{-3}}^{\pr{1}} {h\ps{2}\nsb{0}\ps{1}}^{\pr{1}}\ot  {h\ps{1}}^{\pr{1}}\pr{2}\ns{-1}  {{h\ps{1}}^{\pr{2}}\ns{-3}}^{\pr{2}}\ns{-1} {h\ps{2}\nsb{0}\ps{1}}^{\pr{2}}\ns{-1}   {h\ps{1}}^{\pr{2}}\ns{-2}\nsb{0} h\ps{2}\nsb{0}\ps{2}\nsb{0} \Big) \ot \\
& \Big({h\ps{1}}^{\pr{1}}\pr{2}\ns{0}  {{h\ps{1}}^{\pr{2}}\ns{-3}}^{\pr{2}}\ns{0} {h\ps{2}\nsb{0}\ps{1}}^{\pr{2}}\ns{0}  {h\ps{1}}^{\pr{2}}\ns{-2}\nsb{1} h\ps{2}\nsb{0}\ps{2}\nsb{1} \ot    {h\ps{1}}^{\pr{2}}\ns{-1} h\ps{2}\nsb{0}\ps{3}\Big) \ot  \Big({h\ps{1}}^{\pr{2}}\ns{0}h\ps{2}\nsb{1}\ot h\ps{3}\Big) = \\
& \Big({h\ps{1}}^{\pr{1}}\pr{1} {{h\ps{1}}^{\pr{2}}\ns{-3}}^{\pr{1}} {h\ps{2}\ps{1}\nsb{0}}^{\pr{1}}\ot  {h\ps{1}}^{\pr{1}}\pr{2}\ns{-1}  {{h\ps{1}}^{\pr{2}}\ns{-3}}^{\pr{2}}\ns{-1} {h\ps{2}\ps{1}\nsb{0}}^{\pr{2}}\ns{-1}   {h\ps{1}}^{\pr{2}}\ns{-2}\nsb{0} h\ps{2}\ps{1}\nsb{1}\ns{-2}\nsb{0}h\ps{2}\ps{2}\nsb{0}\nsb{0} \Big) \ot \\
& \Big({h\ps{1}}^{\pr{1}}\pr{2}\ns{0}  {{h\ps{1}}^{\pr{2}}\ns{-3}}^{\pr{2}}\ns{0} {h\ps{2}\ps{1}\nsb{0}}^{\pr{2}}\ns{0}  {h\ps{1}}^{\pr{2}}\ns{-2}\nsb{1} h\ps{2}\ps{1}\nsb{1}\ns{-2}\nsb{1} h\ps{2}\ps{2}\nsb{0}\nsb{1} \ot  \\
& \hspace{8cm}  {h\ps{1}}^{\pr{2}}\ns{-1} h\ps{2}\ps{1}\nsb{1}\ns{-1} h\ps{2}\ps{2}\nsb{1}\ns{-1} h\ps{2}\ps{3}\nsb{0}\Big) \ot \\
& \Big({h\ps{1}}^{\pr{2}}\ns{0} h\ps{2}\ps{1}\nsb{1}\ns{0} h\ps{2}\ps{2}\nsb{1}\ns{0} h\ps{2}\ps{3}\nsb{1}\ot h\ps{3}\Big),
\end{align*}
which, in view of the coassociativity of the comultiplication on $\C{H}$, as well as the (left) $\C{H}$-coaction on $\G{M}$, may be rewritten as
\begin{align*}
& (\D_{\dcc}\ot \Id_{\G{M}\ot \C{H}} )\circ \D_{\dcc} (1\ot h) = \\
& \Big({h\ps{1}}^{\pr{1}}\pr{1} {{h\ps{1}}^{\pr{2}}\ns{-1}}^{\pr{1}} {h\ps{2}\nsb{0}}^{\pr{1}}\ot  {h\ps{1}}^{\pr{1}}\pr{2}\ns{-1}  {{h\ps{1}}^{\pr{2}}\ns{-1}}^{\pr{2}}\ns{-1} {h\ps{2}\nsb{0}}^{\pr{2}}\ns{-1}   {h\ps{1}}^{\pr{2}}\ns{-2}\nsb{0} h\ps{2}\nsb{1}\ns{-2}\nsb{0}h\ps{3}\nsb{0}\nsb{0} \Big) \ot \\
& \Big({h\ps{1}}^{\pr{1}}\pr{2}\ns{0}  {{h\ps{1}}^{\pr{2}}\ns{-1}}^{\pr{2}}\ns{0} {h\ps{2}\nsb{0}}^{\pr{2}}\ns{0}  {h\ps{1}}^{\pr{2}}\ns{0}\ns{-2}\nsb{1} h\ps{2}\nsb{1}\ns{-2}\nsb{1} h\ps{3}\nsb{0}\nsb{1} \ot  \\
& \hspace{8cm}  {h\ps{1}}^{\pr{2}}\ns{0}\ns{-1} h\ps{2}\nsb{1}\ns{-1} h\ps{3}\nsb{1}\ns{-1} h\ps{4}\nsb{0}\Big) \ot \\
& \Big({h\ps{1}}^{\pr{2}}\ns{0}\ns{0} h\ps{2}\nsb{1}\ns{0} h\ps{3}\nsb{1}\ns{0} h\ps{4}\nsb{1}\ot h\ps{5}\Big).
\end{align*}
Now, employing this time \eqref{h-h-h-I}, we arrive at
\begin{align*}
& (\D_{\dcc}\ot \Id_{\G{M}\ot \C{H}} )\circ \D_{\dcc} (1\ot h) = \\
& \Big({h\ps{1}}^{\pr{1}}\ot  {h\ps{1}}^{\pr{2}}\pr{1}\ns{-1} {h\ps{2}}^{\pr{1}}\ns{-1}   {h\ps{1}}^{\pr{2}}\pr{2}\ns{-2}\nsb{0} {h\ps{2}}^{\pr{2}}\ns{-1}\nsb{0}h\ps{3}\nsb{0}\nsb{0} \Big) \ot \\
& \Big({h\ps{1}}^{\pr{2}}\pr{1}\ns{0}  {h\ps{2}}^{\pr{1}}\ns{0} {h\ps{1}}^{\pr{2}}\pr{2}\ns{-2}\nsb{1}  {h\ps{2}}^{\pr{2}}\ns{-1}\nsb{1} h\ps{3}\nsb{0}\nsb{1} \ot  \\
& \hspace{8cm}  {h\ps{1}}^{\pr{2}}\pr{2}\ns{-1} {h\ps{2}}^{\pr{2}}\ns{0}\ns{-1} h\ps{3}\nsb{1}\ns{-1} h\ps{4}\nsb{0}\Big) \ot \\
& \Big({h\ps{1}}^{\pr{2}}\pr{2}\ns{0} {h\ps{2}}^{\pr{2}}\ns{0}\ns{0} h\ps{3}\nsb{1}\ns{0} h\ps{4}\nsb{1}\ot h\ps{5}\Big).
\end{align*}
Next, we apply \eqref{x-h-switch} to get
\begin{align*}
& (\D_{\dcc}\ot \Id_{\G{M}\ot \C{H}} )\circ \D_{\dcc} (1\ot h) = \\
& \Big({h\ps{1}}^{\pr{1}}\ot  {h\ps{1}}^{\pr{2}}\pr{1}\ns{-1}  {h\ps{1}}^{\pr{2}}\pr{2}\ns{-2}\nsb{0} {h\ps{2}}^{\pr{1}}\ns{-1} {h\ps{2}}^{\pr{2}}\ns{-1}\nsb{0}h\ps{3}\nsb{0}\nsb{0} \Big) \ot \\
& \Big({h\ps{1}}^{\pr{2}}\pr{1}\ns{0}  {h\ps{1}}^{\pr{2}}\pr{2}\ns{-2}\nsb{1}  {h\ps{2}}^{\pr{1}}\ns{0}  {h\ps{2}}^{\pr{2}}\ns{-1}\nsb{1} h\ps{3}\nsb{0}\nsb{1} \ot  \\
& \hspace{8cm}  {h\ps{1}}^{\pr{2}}\pr{2}\ns{-1} {h\ps{2}}^{\pr{2}}\ns{0}\ns{-1} h\ps{3}\nsb{1}\ns{-1} h\ps{4}\nsb{0}\Big) \ot \\
& \Big({h\ps{1}}^{\pr{2}}\pr{2}\ns{0} {h\ps{2}}^{\pr{2}}\ns{0}\ns{0} h\ps{3}\nsb{1}\ns{0} h\ps{4}\nsb{1}\ot h\ps{5}\Big),
\end{align*}
and then \eqref{h-h-h-II}, together with the coassociativity of the left $\C{H}$-coaction on $\G{M}$, to obtain
\begin{align*}
& (\D_{\dcc}\ot \Id_{\G{M}\ot \C{H}} )\circ \D_{\dcc} (1\ot h) = \\
& \Big({h\ps{1}}^{\pr{1}}\ot  {h\ps{1}}^{\pr{2}}\pr{1}\ns{-1}  {h\ps{1}}^{\pr{2}}\pr{2}\ns{-1}\nsb{0} h\ps{2}\nsb{0} \Big) \ot \\
& \Big({h\ps{1}}^{\pr{2}}\pr{1}\ns{0}  {h\ps{1}}^{\pr{2}}\pr{2}\ns{-1}\nsb{1}  h\ps{2}\nsb{1}\pr{1}{h\ps{3}}^{\pr{1}} \ot  {h\ps{1}}^{\pr{2}}\pr{2}\ns{0}\ns{-1} h\ps{2}\nsb{1}\pr{2}\ns{-1}{h\ps{3}}^{\pr{2}}\ns{-1} h\ps{4}\nsb{0}\Big) \ot \\
& \Big({h\ps{1}}^{\pr{2}}\pr{2}\ns{0}\ns{0} h\ps{2}\nsb{1}\pr{2}\ns{0} {h\ps{3}}^{\pr{2}}\ns{0} h\ps{4}\nsb{1}\ot h\ps{5}\Big).
\end{align*}
Finally, we apply \eqref{left-comod-coalg} to see \eqref{comultp-h}.
\end{proof}

We shall use the notation $\G{M} \dcc^\s \C{H} $ for the (right) \emph{cocycle double cross coproduct} bialgebra of Proposition \ref{prop-cocycle-double-cross-coprod}. A (left) cocycle double cross coproduct construction $\G{M} {\,}^{\lambda\hspace{-0.15cm}}\dcc \C{H} $ may be pursued similarly.

Once again, a few remarks are in order; before we proceed into the decomposition point of view.

\begin{remark}
If the left coaction \eqref{left-coaction-Hopf-new} is trivial, then the coalgebra $\G{M} \dcc^\s \C{H} $ in Proposition \ref{prop-cocycle-double-cross-coprod} is nothing but $\G{M} \cl^\s \C{H} $, that is, the right-handed version of $\G{M} {\,}^{\lambda\hspace{-0.15cm}}\ccr \C{H} $ in \cite[Prop. 6.3.8]{Majid-book}. If, moreover, \eqref{right-coaction-Hopf-new} is a right coaction, and the comultiplication on $\C{H}$ is trivial, then \eqref{h-h-h-I} indicates that \eqref{sigma-map-Hopf} determines a 2-cocycle in the coalgebra Hochschild cohomology of $\G{M}$ with coefficients in the $\G{M}$-bicomodule $\C{H}$. In short,
\[
\s(h)\in H^2(\G{M},\C{H}),
\]
for any $h\in \C{H}$.
\end{remark}

\begin{remark}
The cocycle double cross coproduct construction does fit into (the right-handed version of) the one given in \cite[Appendix]{Brze97-II}; cross product by an algebra, associated to the dual entwining data $(\C{H},\G{M},\rho,\eta,\rho^\G{M})$ given by
\[
\rho:\G{M}\ot \C{H}\to \C{H}\ot \G{M}, \qquad \rho(x\ot h) := x\ns{-1}h \ot x\ns{0},
\]
and
\[
\rho^\G{M}:\G{M}\ot \G{M} \to \G{M}\ot \G{M}, \qquad \rho^\G{M}(x\ot x') := x\pr{1}x' \ot x\pr{2}.
\]
\end{remark}

Now comes an analogue of Proposition \ref{prop-universal-I}.

\begin{proposition}\label{prop-universal-cross-coprod}
Given an algebra $\G{M}$ and two bialgebras $\C{L},\C{H}$ that fit into 
\[
\xymatrix{
\G{M}  &\ar@{->>}[l]_q \C{L}  \ar@{->>}[r]^p_{\rm coalg} &  \C{H}  
}
\]
in the category of algebras, if  
\begin{equation}\label{iso-L-M-H}
(q\ot p)\circ\D_\C{L}:\C{L}\to\G{M}\ot \C{H}
\end{equation}
is an isomorphism (of algebras), where $\D_\C{L}:\C{L}\to \C{L}\ot \C{L}$ stands for the comultiplication in $\C{L}$, then $\C{L}\cong \G{M} \dcc^\s \C{H}$ as bialgebras. In this case, the maps \eqref{left-coaction-Hopf-new}, \eqref{right-coaction-Hopf-new}, \eqref{delta-map-Hopf}, \eqref{sigma-map-Hopf}  are obtained by
\[
\D_\C{L}(x) = x\pr{1} x\pr{2}\ns{-1}\ot x\pr{2}\ns{0}, \qquad \D_\C{L}(h) = {h\ps{1}}^{\pr{1}} {h\ps{1}}^{\pr{2}}\ns{-1}h\ps{2}\nsb{0}\ot {h\ps{1}}^{\pr{2}}\ns{0}h\ps{2}\nsb{1} h\ps{3}.
\]
\end{proposition}

\begin{proof}
Dual to the proof of Proposition \ref{prop-universal-I}, if \eqref{iso-L-M-H} is an algebra isomorphism, then  there are two algebra morphisms 
\begin{equation}\label{F-map-alg}
F: \G{M}\ot \C{H} \to \C{H}\ot \G{M}
\end{equation}
given by
\[
F(q(\ell\ps{1})\ot p(\ell\ps{2})) = p(\ell\ps{1})\ot q(\ell\ps{2}),
\]
and  
\begin{equation}\label{G-map-alg}
G:\G{M}\ot \C{H} \to \G{M}\ot \G{M}
\end{equation}
given by
\[
G(q(\ell\ps{1})\ot p(\ell\ps{2})) = q(\ell\ps{1})\ot q(\ell\ps{2}),
\]
for any $\ell\in\C{L}$. Accordingly, we define 
\begin{align}\label{maps}
\begin{split}
& \nb:\G{M}\to \C{H}\ot \G{M}, \qquad \nb(x)=x\ns{-1}\ot x\ns{0}:= F(x\ot 1), \\
& \Db:\C{H} \to\C{H}\ot \G{M}, \qquad \Db(h)=h\nsb{0}\ot h\nsb{1}:= F(1\ot h), \\
& \d:\G{M}\to \G{M}\ot \G{M}, \qquad \d(x)=x\pr{1}\ot x\pr{2}:= G(x\ot 1), \\
& \s:\C{H} \to \G{M}\ot \G{M}, \qquad \s(h)=h^{\pr{1}}\ot h^{\pr{2}}:= G(1\ot h),
\end{split}
\end{align}
and hence we can see that
\[
F(x\ot h) = F((x\ot 1)(1\ot h)) = F(x\ot 1)F(1\ot h) = x\ns{-1}h\nsb{0} \ot x\ns{0}h\nsb{1},
\]
and that
\[
G(x\ot h) = G((x\ot 1)(1\ot h)) = G(x\ot 1)G(1\ot h) = x\pr{1}h^{\pr{1}} \ot x\pr{2}h^{\pr{2}}.
\]
Let us next consider the bialgebra structure on $\G{M}\ot\C{H}$ imported by the isomorphism \eqref{iso-L-M-H}. 

Since the comultiplication on $\C{L}$, and the projections $q:\C{L}\to\G{M}$ and $p:\C{L}\to\C{H}$ are morphisms of algebras, then \eqref{iso-L-M-H} is clearly an algebra isomorphism, with the tensor product algebra structure on $\G{M}\ot \C{H}$. 

Next, the comultiplication on $\G{M}\ot \C{H}$, imported from $\C{L}$ via \eqref{iso-L-M-H}, is given by
\begin{align*}
& \D_{\G{M}\ot \C{H}}(q(\ell\ps{1})\ot p(\ell\ps{2})) = \Big(q(\ell\ps{1}\ps{1})\ot p(\ell\ps{1}\ps{2})\Big) \ot \Big(q(\ell\ps{2}\ps{1})\ot p(\ell\ps{2}\ps{2})\Big) = \\
&\Big(q(\ell\ps{1})\ot p(\ell\ps{2})\Big) \ot \Big(q(\ell\ps{3})\ot p(\ell\ps{4})\Big),
\end{align*}
for any $x\ot h =q(\ell\ps{1})\ot p(\ell\ps{2}) \in \G{M}\ot \C{H}$, where $\ell \in \C{L}$. In particular, for
\[
x\ot 1 = q(\ell\ps{1}) \ot p(\ell\ps{2}) = q(\ell\ps{1}) \ve( p(\ell\ps{2}) ) \ot 1 = q(\ell\ps{1}) \ve(\ell\ps{2}) \ot 1 = q(\ell)  \ot 1\in \G{M}\ot \C{H},
\]
where the third equality is a result of the fact that $p:\C{L}\to \C{H}$ is a coalgebra map, we have
\begin{align*}
& x\pr{1} \ot x\pr{2} \ot 1 = G(x\ot 1) \ot 1 = G(q(\ell\ps{1}) \ot 1) \ot p(\ell\ps{2}) = G(q(\ell\ps{1}\ps{1}) \ot p(\ell\ps{1}\ps{2})) \ot p(\ell\ps{2}) =\\
&  q(\ell\ps{1}\ps{1}) \ot q(\ell\ps{1}\ps{2}) \ot p(\ell\ps{2}) = q(\ell\ps{1}) \ot q(\ell\ps{2}) \ot p(\ell\ps{3})
\end{align*}
and hence
\begin{align*}
& x\pr{1} \ot x\pr{2}\ns{-1}\ot x\pr{2}\ns{0} \ot 1 = q(\ell\ps{1})\ot F(q(\ell\ps{2})\ot 1) \ot p(\ell\ps{3}) = \\
& q(\ell\ps{1})\ot F(q(\ell\ps{2}\ps{1})\ot p(\ell\ps{2}\ps{2})) \ot p(\ell\ps{3})= q(\ell\ps{1})\ot p(\ell\ps{2}\ps{1})\ot q(\ell\ps{2}\ps{2}) \ot p(\ell\ps{3}) = \\
& q(\ell\ps{1})\ot p(\ell\ps{2})\ot q(\ell\ps{3}) \ot p(\ell\ps{4}).
\end{align*}
That is,
\[
\D_{\G{M}\ot \C{H}}(x\ot 1) = \Big(x\pr{1}\ot x\pr{2}\ns{-1}\Big)\ot \Big(x\pr{2}\ns{0} \ot 1\Big).
\]
Along the same lines, for $1\ot h = q(\ell\ps{1})\ot p(\ell\ps{2}) \in \G{M}\ot \C{H}$, using once again the fact that $p:\C{L}\to \C{H}$ is a coalgebra map,
\[
1 \ot h\ps{1}\ot h\ps{2}\ot h\ps{3} = q(\ell\ps{1})\ot p(\ell\ps{2})\ot p(\ell\ps{3})\ot p(\ell\ps{4}),
\] 
and hence
\begin{align*}
& {h\ps{1}}^{\pr{1}}\ot {h\ps{1}}^{\pr{2}}\ot h\ps{2}\ot h\ps{3} = G(q(\ell\ps{1})\ot p(\ell\ps{2}))\ot p(\ell\ps{3})\ot p(\ell\ps{4}) = \\
& G(q(\ell\ps{1}\ps{1})\ot p(\ell\ps{1}\ps{2}))\ot p(\ell\ps{2})\ot p(\ell\ps{3}) = q(\ell\ps{1}\ps{1})\ot q(\ell\ps{1}\ps{2})\ot p(\ell\ps{2})\ot p(\ell\ps{3}) = \\
& q(\ell\ps{1})\ot q(\ell\ps{2})\ot p(\ell\ps{3})\ot p(\ell\ps{4}).
\end{align*}
Accordingly,
\begin{align*}
& {h\ps{1}}^{\pr{1}}\ot {h\ps{1}}^{\pr{2}}\ns{-1}h\ps{2}\nsb{0}\ot {h\ps{1}}^{\pr{2}}\ns{0}h\ps{2}\nsb{1}\ot h\ps{3} = q(\ell\ps{1})\ot F(q(\ell\ps{2})\ot p(\ell\ps{3}))\ot p(\ell\ps{4}) = \\
& q(\ell\ps{1})\ot F(q(\ell\ps{2}\ps{1})\ot p(\ell\ps{2}\ps{2}))\ot p(\ell\ps{3}) = q(\ell\ps{1})\ot p(\ell\ps{2}\ps{1})\ot q(\ell\ps{2}\ps{2})\ot p(\ell\ps{3}) = \\
& q(\ell\ps{1})\ot p(\ell\ps{2})\ot q(\ell\ps{3})\ot p(\ell\ps{4}).
\end{align*}
In other words,
\[
\D_{\G{M}\ot \C{H}}(1\ot h) = \Big({h\ps{1}}^{\pr{1}}\ot {h\ps{1}}^{\pr{2}}\ns{-1}h\ps{2}\nsb{0}\Big)\ot \Big({h\ps{1}}^{\pr{2}}\ns{0}h\ps{2}\nsb{1}\ot h\ps{3}\Big).
\]
By the multiplicativity of the maps, we have thus proved that the comultiplication on $\G{M}\ot \C{H}$ transferred from $\C{L}$ coincides with that of \eqref{comult-cocycly-cross-coprod-comultp}.
\end{proof}

\subsection{Bicocycle Double Cross Product Bialgebras}
\label{subsect-Bicocycle-double-cross-coproduct-bialgebra}~

Given two coalgebras $(\G{M},e)$ and $(\C{H},1)$, with distinguished group-likes $e \in\G{M}$ and $1\in \C{H}$, in the present subsection we shall present the construction of a bialgebra $\G{M} {\,}_{\g\hspace{-0.1cm}}\bowtie_\t \C{H} :=\G{M}\ot \C{H}$ that fits into the picture
\[
\xymatrix{
\G{M} \ar@{^{(}->}[r]_{i\,\,\,\,\,\,\,\,\,\,\,\,\,\,\,}  & \G{M} {\,}_{\g\hspace{-0.1cm}}\bowtie_\t \C{H}    &  \ar@{_{(}->}[l]^{\,\,\,\,\,\,\,\,\,\,\,\,\,\,\, j}\C{H},
}
\]
in the category of coalgebras. In other words, both $\G{M}$ and $\C{H}$ are accommodated merely as subcoalgebras in $\G{M} {\,}_{\g\hspace{-0.1cm}}\bowtie_\t \C{H}$.

Accordingly, in addition to 
\begin{align}
& \vp:\C{H}\ot \G{M} \to \G{M}, \quad h\ot x\mapsto \vp(h,x), \label{left-action-Hopf-2} \\
& \psi:\C{H}\ot \G{M} \to \C{H}, \quad h\ot x\mapsto \psi(h,x), \label{right-action-Hopf-2}  \\
& \phi:\G{M}\ot \G{M} \to \G{M}, \quad x\ot x'\mapsto \phi(x,x')=:x\cdot x', \label{phi-map-Hopf-2} \\
& \t:\G{M}\ot \G{M} \to \C{H}, \quad x\ot x'\mapsto \t(x,x'), \label{theta-map-Hopf-2}
\end{align}
we shall need two more (coalgebra) maps
\begin{equation}\label{mu-map-Hopf}
\mu:\C{H}\ot \C{H}\to \C{H}, \qquad h\ot h' \mapsto \mu(h, h') =: h\ast h'
\end{equation}
and 
\begin{equation}\label{gamma-map-Hopf}
\g:\C{H}\ot \C{H}\to \G{M}, \qquad h\ot h' \mapsto \g(h, h')=:\g(h\ot h').
\end{equation}
More precisely, we shall begin with the following data.

\begin{definition}
Let $(\G{M},e)$ and $(\C{H},1)$ be two coalgebras with distinguished group-likes $e \in\G{M}$, and $1\in \C{H}$. Then the pair $(\G{M},\C{H})$, equipped with the coalgebra maps \eqref{left-action-Hopf-2}, \eqref{right-action-Hopf-2}, \eqref{phi-map-Hopf-2}, \eqref{theta-map-Hopf-2}, \eqref{mu-map-Hopf}, and \eqref{gamma-map-Hopf} that satisfy  
\begin{align}\label{basics}
\begin{split}
& \vp(1,x) = x, \qquad \psi(h,e)=h,\qquad  \vp(h,e) = \ve(h)e, \qquad \psi(1,x)=\ve(x)1,\\
&\t(e,e) = 1, \qquad\mu(1,1)=1, \qquad \phi(e,e)=e, \qquad \g(1,1) = e
\end{split}
\end{align}
is called a \emph{bicocycle double cross product pair}.
\end{definition}

\begin{proposition}\label{prop-bicocycle-double-cross-prod}
Let $(\G{M}, \C{H})$ be a bicocycle double cross product pair. Then, $\C{G}:=\G{M}\ot \C{H}$ is a bialgebra through
\begin{equation}\label{mult-bicocycly-cross-prod-multp}
(x\ot h)(x'\ot h') := x\ps{1}\cdot [\vp(h\ps{1}, x'\ps{1})\cdot\g(\psi(h\ps{2},x'\ps{2}),h'\ps{1})] \ot [\t(x\ps{2},\vp(h\ps{3}, x'\ps{3}))\ast\psi(h\ps{4},x'\ps{4})]\ast h'\ps{2},
\end{equation}
and
\begin{equation}\label{comultp-cross-prod-II}
\D(x\ot h) := (x\ps{1}\ot h\ps{1}) \ot (x\ps{2}\ot h\ps{2}),
\end{equation}
with the unit $e\ot 1 \in \G{M}\ot \C{H}$, and the counit given by $\ve(x\ot h):= \ve(x)\ve(h)$, if and only if \eqref{left-action-Hopf}, \eqref{right-action-Hopf}, \eqref{phi-map-Hopf}, \eqref{theta-map-Hopf}, \eqref{mu-map-Hopf}, \eqref{gamma-map-Hopf}  are subject to
\begin{align}
& e\cdot x = x = x\cdot e, \qquad 1\ast h = h = h \ast 1 \label{phi-mu-map-Hopf-II} \\
& \t(x,e) = \ve(x)1 = \t(e,x), \qquad \g(h,1) = \ve(h)e = \g(1,h) \label{theta-gamma-map-Hopf-II} \\
& \vp(h\ps{1}, x'\ps{1}\cdot x''\ps{1})\cdot \g(\psi(h\ps{2},x'\ps{2}\cdot x''\ps{2}),\t(x\ps{3},x''\ps{3})\ast h'') = \label{left-action-Hopf-on-multp-III} \\
& \hspace{3cm} \vp(h\ps{1},x'\ps{1})\cdot[\vp(\psi(h\ps{2},x'\ps{2}),x''\ps{1})\cdot \g(\psi(\psi(h\ps{3},x'\ps{3}),x''\ps{2}),h'')] , \notag\\
& \psi(h,x'\ps{1}\cdot x''\ps{1}) \ast[\t(x'\ps{2}, x''\ps{2})\ast h''] = \label{right-action-theta-comp-Hopf-III} \\
& \hspace{3cm}\t(\vp(h\ps{1}, x'\ps{1}), [\vp(\psi(h\ps{2}, x'\ps{2}), x''\ps{1}))\ast\psi(\psi(h\ps{3},x'\ps{3}), x''\ps{2})]\ast h'', \notag\\
& \t(x\cdot\g(h\ps{1},h'\ps{1}), \vp(h\ps{2}\ast h'\ps{2}, x''\ps{1}))\ast\psi(h\ps{3}\ast h'\ps{3}, x''\ps{2}) = \label{right-action-Hopf-on-multp-III} \\
& \hspace{3cm} [\t(x,\vp(h\ps{1},\vp(h'\ps{1},x''\ps{1})))\ast \psi(h\ps{2},\vp(h'\ps{2},x''\ps{2}))]\ast \psi(h'\ps{3},x''\ps{3}),  \notag \\
& [x\cdot\g(h\ps{1},h'\ps{1})]\cdot\vp(h\ps{2}\ast h'\ps{2}, x'') = x\cdot [\vp(h\ps{1},\vp(h'\ps{1}, x''\ps{1}))\cdot \g(\psi(h\ps{2},\vp(h'\ps{2},x''\ps{2})), \psi(h'\ps{3},x''\ps{3}))], \label{M-is-left-H-module-III} \\
& x\cdot (x'\cdot x'') = (x\ps{1}\cdot x'\ps{1})\cdot [\vp(\t(x\ps{2}, x'\ps{2}), x''\ps{1})\cdot \g(\psi(\t(x\ps{3}, x'\ps{3}), x''\ps{2}), h'')], \label{phi-map-Hopf-assoc-III} \\
& \t(x, x'\ps{1}\cdot x''\ps{1})\ast [\t(x'\ps{2}, x''\ps{2})\ast h''] = [\t(x\ps{1}\cdot x'\ps{1}, \vp(\t(x\ps{2}, x'\ps{2}), x''\ps{1}))\ast \psi(\t(x\ps{3}, x'\ps{3}), x''\ps{2})] \ast h'',  \label{theta-map-Hopf-cocycle-III} \\
& [x\cdot \g(h\ps{1},h'\ps{1})] \cdot \g(h\ps{2}\ast h'\ps{2},h'') = x\cdot [\vp(h\ps{1}, \g(h'\ps{1}, h''\ps{1})) \cdot \g(\psi(h\ps{2},\g(h'\ps{2}, h''\ps{2})), h'\ps{3}\ast h''\ps{3})] , \label{gamma-map-Hopf-cocycle-III}\\
& (h\ast h') \ast h'' = [\t(x,\vp(h\ps{1},\g(h'\ps{1}, h''\ps{1})))\ast \psi(h\ps{2},\g(h'\ps{2}, h''\ps{2}))]\ast (h'\ps{3}\ast h''\ps{3}), \label{gamma-map-Hopf-assoc-III}\\
& \psi(h\ps{2},x\ps{2}) \ot \vp(h\ps{1}, x\ps{1}) = \psi(h\ps{1},x\ps{1}) \ot \vp(h\ps{2}, x\ps{2}), \label{left-right-action-comultp-II} \\
& \t(x\ps{2},x'\ps{2}) \ot x\ps{1}\cdot x'\ps{1} =  \t(x\ps{1},x'\ps{1}) \ot x\ps{2}\cdot x'\ps{2},\label{theta-phi-comp-II} \\
& h\ps{2}\ast h'\ps{2} \ot \g(h\ps{1},h'\ps{1}) = h\ps{1}\ast h'\ps{1} \ot \g(h\ps{2},h'\ps{2}), \label{gamma-mu-comp-II}
\end{align}
for any $x,x',x'' \in \G{M}$, and any $h,h',h'' \in \C{H}$.
\end{proposition}

\begin{proof}
We shall begin with the unitality, that is, $e\ot 1 \in \G{M}\ot \C{H}$ being the unit of the multiplication \eqref{mult-bicocycly-cross-prod-multp}, and then we shall proceed to the associativity. We observe at once that
\[
(e\ot 1)(x\ot 1) = e \cdot (x\ps{1} \cdot e)\ot (\t(e,x\ps{2}) \ast 1) \ast 1,
\]
and that
\[
(x\ot 1)(e\ot 1) = x\ps{1} \cdot e^2\ot (\t(x\ps{2},e) \ast 1) \ast 1.
\]
Accordingly, 
\[
(e\ot 1)(x\ot 1) = (x\ot 1) = (x\ot 1)(e\ot 1)
\]
if and only if 
\[
e\cdot (x \cdot e) = x = x \cdot e^2
\]
and
\[
(\t(x,e) \ast 1) \ast 1 = \ve(x)1 = (\t(e,x) \ast 1) \ast 1.
\]
On the other hand, 
\[
(e\ot h)(e\ot 1) =  e \cdot (e \cdot \g(h\ps{1},1))\ot (1 \ast h\ps{2}) \ast 1 
\]
while
\[
(e\ot 1)(e\ot h) = e \cdot (e \cdot \g(1,h\ps{1})) \ot 1^2 \ast h\ps{2}.
\]
Thus, 
\[
(e\ot h)(e\ot 1) = e \ot h = (e\ot 1)(e\ot h)
\]
if and only if
\[
1^2\ast h = h = (1 \ast h) \ast 1,
\]
and
\[
e \cdot (e \cdot \g(1,h)) = \ve(h)e = e \cdot (e \cdot \g(h,1)).
\]
As a result, in view of \eqref{basics}, $e\ot 1 \in \G{M}\ot\C{H}$ is the identity element of the multiplication \eqref{mult-bicocycly-cross-prod-multp} if and only if \eqref{phi-mu-map-Hopf-II} and \eqref{theta-gamma-map-Hopf-II} hold.

We now come to the (mixed) associativity conditions for \eqref{mult-bicocycly-cross-prod-multp}.  As for $x\ot 1, x'\ot 1, x''\ot h''\in \G{M}\ot \C{H}$, we have on one hand
\begin{align*}
& (x\ot 1)[(x'\ot 1)(x''\ot h'')] = (x\ot 1)(x'\ps{1}\cdot x''\ps{1} \ot \t(x'\ps{2}, x''\ps{2})\ast h'') = \\
& x\ps{1}\cdot [x'\ps{1}\cdot x'\ps{1}] \ot \t(x\ps{2}, x'\ps{2}\cdot x''\ps{2})\ast  [\t(x'\ps{3}, x''\ps{3})\ast h''],
\end{align*}
while on the other hand,
\begin{align*}
& [(x\ot 1)(x'\ot 1)](x''\ot h'') = (x\ps{1}\cdot x'\ps{1}\ot \t(x\ps{2},x'\ps{2}))(x''\ot h'') = \\
&  [x\ps{1}\cdot x'\ps{1}]\cdot [\vp(\t(x\ps{3}, x'\ps{3}), x''\ps{1})\cdot \g(\psi(\t(x\ps{4}, x'\ps{4}), x''\ps{2}), h''\ps{1})] \ot \\
& \hspace{3cm} [\t(x\ps{2}\cdot x'\ps{2}, \vp(\t(x\ps{5}, x'\ps{5}), x''\ps{3}))\ast \psi(\t(x\ps{6}, x'\ps{6}), x''\ps{4})] \ast h''\ps{2}.
\end{align*}
Accordingly,
\[
(x\ot 1)[(x'\ot 1)(x''\ot h'')]  =  [(x\ot 1)(x'\ot 1)](x''\ot h'')
\]
yields \eqref{phi-map-Hopf-assoc-III} and \eqref{theta-map-Hopf-cocycle-III}. 

We, then, proceed onto the mixed associativity condition for $x\ot h, e\ot h', e\ot h''\in \G{M}\ot \C{H}$. To this end, we have
\begin{align*}
&[(x\ot h)(e\ot h')](e\ot h'') = (x\cdot\g(h\ps{1},h'\ps{1}) \ot h\ps{2}\ast h'\ps{2})(e\ot h'') = \\
&  [x\cdot\g(h\ps{1},h'\ps{1})]\cdot \g(h\ps{2}\ast h'\ps{2}, h''\ps{1})\ot [h\ps{3}\ast h'\ps{3}]\ast h''\ps{2},
\end{align*}
and
\begin{align*}
& (x\ot h)[(e\ot h')(e\ot h'')] = (x\ot h)(\g(h'\ps{1}, h''\ps{1}) \ot h'\ps{2} \ast h''\ps{2}) = \\
& x\ps{1}\cdot [\vp(h\ps{1}, \g(h'\ps{1}, h''\ps{1})) \cdot \g(\psi(h\ps{2},\g(h'\ps{2}, h''\ps{2})), h'\ps{5}\ast h''\ps{5})] \ot \\
& \hspace{3cm} [\t(x\ps{2},\vp(h\ps{3},\g(h'\ps{3}, h''\ps{3})))\ast \psi(h\ps{4},\g(h'\ps{4}, h''\ps{4}))]\ast [h'\ps{6}\ast h''\ps{6}].
\end{align*}
Hence,
\[
[(x\ot h)(e\ot h')](e\ot h'')= (x\ot h)[(e\ot h')(e\ot h'')]
\]
implies \eqref{gamma-map-Hopf-cocycle-III} and \eqref{gamma-map-Hopf-assoc-III}. 

Next, we observe for $x\ot h, e\ot h', x'' \ot 1\in \G{M}\ot \C{H}$ that
\begin{align*}
& [(x\ot h)(e\ot h')](x''\ot 1) = (x\cdot \g(h\ps{1},h'\ps{1}) \ot h\ps{2}\ast h'\ps{2})(x''\ot 1) = \\
&[x\ps{1}\cdot \g(h\ps{1},h'\ps{1})] \cdot \vp(h\ps{3}\ast h'\ps{3}, x''\ps{1}) \ot \t(x\ps{2}\cdot \g(h\ps{2},h'\ps{2}),\vp(h\ps{4}\ast h'\ps{4}, x''\ps{2}))\ast \psi(h\ps{5}\ast h'\ps{5}, x''\ps{3}),
\end{align*}
and that
\begin{align*}
& (x\ot h)[(e\ot h')(x''\ot 1)] =(x\ot h)(\vp(h'\ps{1},x''\ps{1})\ot \psi(h'\ps{2},x''\ps{2})) = \\
& x\ps{1}\cdot [\vp(h\ps{1},\vp(h'\ps{1},x''\ps{1}))\cdot \g(\psi(h\ps{2},\vp(h'\ps{2},x''\ps{2})),\psi(h'\ps{5},x''\ps{5}))] \ot \\
& \hspace{3cm} [\t(x\ps{2},\vp(h\ps{3},\vp(h'\ps{3},x''\ps{3})))\ast \psi(h\ps{4},\vp(h'\ps{4},x''\ps{4}))]\ast \psi(h'\ps{6},x''\ps{6}).
\end{align*}
Therefore, we conclude \eqref{right-action-Hopf-on-multp-III} and \eqref{M-is-left-H-module-III}  from
\[
[(x\ot h)(e\ot h')](x''\ot 1)  = (x\ot h)[(e\ot h')(x''\ot 1)]. 
\]
Lastly, for $e\ot h, x'\ot 1, x''\ot h''\in \G{M}\ot\C{H}$ we consider the associativity through
\begin{align*}
& (e \ot h)[(x' \ot 1)(x'' \ot h'')] = (e\ot h)(x'\ps{1}\cdot x''\ps{1}\ot \t(x'\ps{2},x''\ps{2})\ast h'') = \\
&\vp(h\ps{1},x'\ps{1}\cdot x''\ps{1})\cdot \g(\psi(h\ps{2},x'\ps{2}\cdot x''\ps{2}),\t(x'\ps{4},x''\ps{4})\ast h''\ps{1}) \ot \\
& \hspace{6cm} \psi(h\ps{3},x'\ps{3}\cdot x''\ps{3})\ast [\t(x'\ps{5},x''\ps{5})\ast h''\ps{2}],
\end{align*}
and
\begin{align*}
& [(e \ot h)(x' \ot 1)](x'' \ot h'') = (\vp(h\ps{1},x'\ps{1})\ot \psi(h\ps{2},x'\ps{2}))(x'' \ot h'')  = \\
& \vp(h\ps{1},x'\ps{1})\cdot[\vp(\psi(h\ps{3},x'\ps{3}),x''\ps{1})\cdot \g(\psi(\psi(h\ps{4},x'\ps{4}),x''\ps{2}),h''\ps{1})] \ot \\
& \hspace{3cm} [\t(\vp(h\ps{2},x'\ps{2}),\vp(\psi(h\ps{5},x'\ps{5}),x''\ps{3}))\ast \psi(\psi(h\ps{6},x'\ps{6}),x''\ps{4})] \ast h''\ps{2},
\end{align*}
and we conclude that 
\[
(e \ot h)[(x' \ot 1)(x'' \ot h'')] =[(e \ot h)(x' \ot 1)](x'' \ot h'') 
\]
gives \eqref{left-action-Hopf-on-multp-III} and \eqref{right-action-theta-comp-Hopf-III}. 

Finally, we check the multiplicativity of the comultiplication $\D:\C{G}\to \C{G}\ot \C{G}$. On one hand we have 
\begin{align*}
& \D((x\ot h)(x'\ot h')) = \\
& \D\Big(x\ps{1}\cdot \vp(h\ps{1}, x'\ps{1})\cdot \g(\psi(h\ps{2},x'\ps{2}),h'\ps{1}) \ot \t(x\ps{2},\vp(h\ps{3}, x'\ps{3}))\ast \psi(h\ps{4},x'\ps{4})\ast h'\ps{2}\Big) = \\
& \Big(x\ps{1}\cdot \vp(h\ps{1}, x'\ps{1})\cdot \g(\psi(h\ps{3},x'\ps{3}),h'\ps{1}) \ot \t(x\ps{3}, \vp(h\ps{5}, x'\ps{5}))\ast \psi(h\ps{7},x'\ps{7})\ast h'\ps{3}\Big) \ot \\
& \hspace{2cm} \Big(x\ps{2}\cdot\vp(h\ps{2}, x'\ps{2})\cdot\g(\psi(h\ps{4},x'\ps{4}),h'\ps{2}) \ot \t(x\ps{4}, \vp(h\ps{6}, x'\ps{6}))\ast\psi(h\ps{8},x'\ps{8})\ast h'\ps{4}\Big),
\end{align*}
and on the other hand
\begin{align*}
& \D(x\ot h)\D(x'\ot h') = \Big((x\ps{1}\ot h\ps{1}) \ot (x\ps{2}\ot h\ps{2})\Big)\Big((x'\ps{1}\ot h'\ps{1}) \ot (x'\ps{2}\ot h'\ps{2})\Big) = \\
& (x\ps{1}\ot h\ps{1})(x'\ps{1}\ot h'\ps{1}) \ot (x\ps{2}\ot h\ps{2})(x'\ps{2}\ot h'\ps{2}) = \\
& \Big(x\ps{1}\vp(h\ps{1}, x'\ps{1})\g(\psi(h\ps{2},x'\ps{2}),h'\ps{1}) \ot \t(x\ps{2}, \vp(h\ps{3}, x'\ps{3}))\psi(h\ps{4},x'\ps{4})h'\ps{2}\Big) \ot \\
& \hspace{2cm} \Big(x\ps{3}\vp(h\ps{5}, x'\ps{5})\g(\psi(h\ps{6},x'\ps{6}),h'\ps{3}) \ot \t(x\ps{4}, \vp(h\ps{7}, x'\ps{7}))\psi(h\ps{8},x'\ps{8})h'\ps{4}\Big).
\end{align*}
Thus, we see at once that
\[
\D((x\ot 1)(e\ot h')) = \D(x\ot 1)\D(e\ot h').
\]
Accordingly, \eqref{comultp-cross-prod-II} is multiplicative if and only if
\begin{equation}\label{h-h'-Delta}
\D((e\ot h)(e\ot h')) = \D(e\ot h)\D(e\ot h'),
\end{equation}
as well as
\begin{equation}\label{h-x'-Delta}
\D((e\ot h)(x'\ot 1)) = \D(e\ot h)\D(x'\ot 1)
\end{equation}
and
\begin{equation}\label{x-x'-Delta}
\D((x\ot 1)(x'\ot 1)) = \D(x\ot 1)\D(x'\ot 1).
\end{equation}
Applying the counits on the left-most and the right-most components of \eqref{h-x'-Delta} yields \eqref{left-right-action-comultp-II}, while its application to \eqref{x-x'-Delta} gives \eqref{theta-phi-comp-II}. Finally, we obtain \eqref{gamma-mu-comp-II} applying the counits on the left-most and the right-most components of \eqref{h-h'-Delta}. Conversely, \eqref{left-right-action-comultp-II}, \eqref{theta-phi-comp-II}, and \eqref{gamma-mu-comp-II} together imply the multiplicativity of \eqref{comultp-cross-prod-II}.
\end{proof}

We shall denote the bialgebra of Proposition \ref{prop-bicocycle-double-cross-prod} by $\G{M} {\,}_{\g\hspace{-0.1cm}}\bowtie_\t \C{H} :=\G{M}\ot \C{H}$, and call it the \emph{bicocycle double cross product} bialgebra of the coalgebras $\G{M}$ and $\C{H}$.

\begin{remark}
The dual construction $\G{M} \,{}^{\lambda\hspace{-.15cm}}\dcc^\s \C{H}$ of a \emph{bicocycle double cross coproduct} bialgebra, generalizing $\G{M} \dcc^\s \C{H}$ of the previous subsection, may be built similarly.
\end{remark}

We now record here a comparison with the earlier constructions.

\begin{remark}
Let us note that if \eqref{left-action-Hopf-2} is a (left) action, then its being a map of coalgebras is equivalent to $\G{M}$ being a (left) $\C{H}$-module coalgebra. Similarly, if  \eqref{right-action-Hopf-2} happens to be a (right) action, then $\C{H}$ becomes a (right) $\G{M}$-module coalgebra. In this case, \eqref{left-action-Hopf-on-multp-III}, \eqref{right-action-Hopf-on-multp-III}, and \eqref{left-right-action-comultp-II} are nothing but (7.7), (7.8) and (7.9) of \cite[Def. 7.2.1]{Majid-book}. If, furthermore, both \eqref{theta-map-Hopf-2} and \eqref{gamma-map-Hopf} are trivial, then Proposition \ref{prop-bicocycle-double-cross-prod} above coincides with  \cite[Thm. 7.2.2]{Majid-book}.
\end{remark}

\begin{remark}
If \eqref{left-action-Hopf-2} and \eqref{gamma-map-Hopf} are trivial, then the algebra $\G{M} {\,}_{\g\hspace{-0.1cm}}\bowtie_\t \C{H}$ given in Proposition \ref{prop-bicocycle-double-cross-prod} is nothing but $\G{M} \ltimes_\theta\C{H}$ of \cite[Prop. 6.3.7]{Majid-book}, since in this case (6.27) follows from \eqref{right-action-Hopf-on-multp-III}, (6.28) from \eqref{right-action-theta-comp-Hopf-III}, and (6.29) from \eqref{theta-map-Hopf-cocycle-III}. If, furthermore, \eqref{right-action-Hopf-2} is a right action and the multiplication \eqref{mu-map-Hopf} on $\C{H}$ is trivial; that is, given by the addition (in other words, $\C{H}$ is regarded only as a vector space), then \eqref{theta-map-Hopf-cocycle-III} indicates that \eqref{theta-map-Hopf-2} is a 2-cocycle in the algebra Hochschild cohomology of $\G{M}$ with coefficients in $\C{H}$. In short,
\[
\t\in H^2(\G{M},\C{H}).
\]
Similarly, if \eqref{right-action-Hopf-2} and \eqref{theta-map-Hopf-2} are trivial, then the bicocycle double cross product algebra $\G{M} {\,}_{\g\hspace{-0.1cm}}\bowtie_\t \C{H}$ of Proposition \ref{prop-bicocycle-double-cross-prod} becomes $\G{M} {\,}_{\g\hspace{-0.1cm}}\rtimes \C{H}$ of \cite[Prop. 6.3.2]{Majid-book}. If, moreover, \eqref{left-action-Hopf-2} is trivial, and the multiplication \eqref{phi-map-Hopf-2} on $\G{M}$ is trivial, then \eqref{gamma-map-Hopf-cocycle-III} yields
\[
\g\in H^2(\C{H},\G{M}).
\]
\end{remark}

\begin{remark}
If \eqref{theta-map-Hopf-2} is trivial, then $\G{M} {\,}_{\g\hspace{-0.1cm}}\bowtie_\t \C{H} = \G{M} {\,}_{\g\hspace{-0.1cm}}\bowtie\C{H}$ is precisely the unified product in \cite[Thm. 2.4]{AgorMili11}. If, on the other hand, \eqref{gamma-map-Hopf} is trivial, then $\G{M} {\,}_{\g\hspace{-0.1cm}}\bowtie_\t \C{H} = \G{M} \bowtie_\t \C{H}$ of Proposition \ref{prop-cocycle-double-cross-prod}, which is the right-handed counterpart of the one in \cite[Thm. 2.4]{AgorMili11}. In other words, the bicocycle double cross product construction collects, under a sigle roof, the right handed and the left handed unified products.
\end{remark}

We shall conclude the present subsection with an upgrade of Corollary \ref{prop-U-cocycle-double-cross-sum} to the level of bicocycle double cross sum Lie algebras. However, to this end we shall need the following analogue of Proposition \ref{prop-universal-I}.

\begin{proposition}\label{prop-universal-II}
Given two coalgebras $\G{M}$ and $\C{H}$, a bialgebra $\C{G}$, and the maps
\[
\xymatrix{
\G{M} \ar@{^{(}->}[r]_i  & \C{G}    &  \ar@{_{(}->}[l]^j \C{H},
}
\]
in the category of coalgebras, if $\mu\circ (i\ot j):\G{M}\ot \C{H} \to\C{G}$ is an isomorphism (of coalgebras), where $\mu:\C{G}\ot \C{G}\to \C{G}$ denotes the multiplication in $\C{G}$, then $\C{G}\cong \G{M} {\,}_{\g\hspace{-0.1cm}}\bowtie_\t \C{H}$ as bialgebras. In this case, the maps \eqref{left-action-Hopf-2} - \eqref{gamma-map-Hopf} are obtained by 
\[
hx = \vp(h\ps{1}, x\ps{1}) \psi(h\ps{2},x\ps{2}), \qquad xx'=(x\ps{1}\cdot x'\ps{1}) \t(x\ps{2},x'\ps{2}), \qquad hh' = \g(h\ps{1}, h'\ps{1}) (h\ps{2}\ast h'\ps{2}).
\]
\end{proposition}

\begin{proof}
Along the lines of \cite[Thm. 7.2.3]{Majid-book}, there are three (coalgebra) morphisms 
\begin{equation}\label{f-map-II}
f:\C{H}\ot \G{M} \to \G{M}\ot \C{H}
\end{equation}
given by
\[
j(h)i(x) = \Big(\mu\circ (i\ot j) \circ f \Big)(h\ot x),
\]
and  
\begin{equation}\label{g-map-II}
g:\G{M}\ot \G{M} \to \G{M}\ot \C{H}
\end{equation}
given by
\[
i(x)i(x') = \Big(\mu\circ (i\ot j) \circ g \Big)(x\ot x'),
\]
together with
\begin{equation}\label{r-map-II}
r:\C{H}\ot \C{H} \to \G{M}\ot \C{H}
\end{equation}
given by
\[
j(h)j(h') = \Big(\mu\circ (i\ot j) \circ r \Big)(h\ot h').
\]
Accordingly, we define 
\begin{align}\label{maps-I-II}
\begin{split}
& \vp:\C{H}\ot \G{M}\to \G{M}, \qquad \vp:= (\Id\ot \ve)\circ f, \\
& \psi:\C{H}\ot \G{M}\to \C{H}, \qquad \psi:= (\ve\ot \Id)\circ f, \\
& \phi:\G{M}\ot \G{M}\to \G{M}, \qquad \phi:= (\Id\ot \ve)\circ g, \\
& \t:\G{M}\ot \G{M}\to \C{H}, \qquad \t:= (\ve\ot \Id)\circ g, \\
& \mu:\C{H}\ot \C{H}\to \C{H}, \qquad \mu:= (\ve\ot \Id)\circ r, \\
& \g:\C{H}\ot \C{H}\to \G{M}, \qquad \g:= (\Id\ot \ve)\circ r.
\end{split}
\end{align}
The explicit expressions of the maps \eqref{maps-I-II} are obtained from the fact that \eqref{f-map-II}, \eqref{g-map-II}, and \eqref{r-map-II} are coalgebra morphisms. More precisely, the application of $\Id \ot \ve \ot \ve \Id:\G{M}\ot \C{H} \ot \G{M}\ot \C{H} \to \G{M}\ot \C{H}$ to
\[
\D_{\G{M}\ot \C{H}} \circ f = (f\ot f) \circ \D_{\C{H}\ot \G{M}}
\]
yields
\[
f(h\ot x) = \vp(h\ps{1},x\ps{1}) \ot \psi(h\ps{2},x\ps{2}),
\]
while its application to
\[
\D_{\G{M}\ot \C{H}} \circ g = (g\ot g) \circ \D_{\G{M}\ot \G{M}}
\]
gives
\[
g(x\ot x') = \phi(x\ps{1},x'\ps{1}) \ot \t(x\ps{2},x'\ps{2}).
\]
Finally, the application of $\Id \ot \ve \ot \ve \Id:\G{M}\ot \C{H} \ot \G{M}\ot \C{H} \to \G{M}\ot \C{H}$ to 
\[
\D_{\G{M}\ot \C{H}} \circ r = (r\ot r) \circ \D_{\G{M}\ot \G{M}}
\]
implies
\[
r(h\ot h') = \g(h\ps{1},h'\ps{1}) \ot \mu(h\ps{2},h'\ps{2}).
\]
Therefore, the multiplication on $\G{M}\ot \C{H}$ which is borrowed from $\C{G}$ is given by
\begin{align*}
& (e\ot h)(x\ot 1) := \vp(h\ps{1}, x\ps{1}) \ot \psi(h\ps{2},x\ps{2}), \\
& (x\ot 1)(x'\ot 1) := \phi(x\ps{1}, x'\ps{1})\ot \t(x\ps{2},x'\ps{2}), \\
& (e\ot h)(e\ot h') := \g(h\ps{1}, h'\ps{1})\ot \mu(h\ps{2},h'\ps{2}),
\end{align*}
with the unit being $e\ot 1\in \G{M}\ot \C{H}$, while the coalgebra structure on $\G{M}\ot \C{H}$ is defined to be the tensor product coalgebra structure. 

To conclude; we impose the associativity and the unitality of the algebra structure on $\G{M}\ot \C{H}$, as well as the multiplicativity of the coalgebra structure maps, to derive \eqref{basics}, while \eqref{phi-mu-map-Hopf-II}-\eqref{gamma-mu-comp-II} are all satisfied by the maps of \eqref{maps-I-II}. 
\end{proof}

We are now ready to discuss the universal enveloping algebra $U(\G{g})$ of a bicocycle double cross sum Lie algebra $\G{g}:=\G{m}{\,}_{\g\hspace{-0.1cm}}\bowtie_\t\G{h}$.

\begin{corollary}\label{prop-U-bicocycle-double-cross-sum}
Given the cocycle double cross sum Lie algebra $\G{g}:=\G{m} {\,}_{\g\hspace{-0.1cm}}\bowtie_\t\G{h}$, let $\{\xi_1,\ldots,\xi_n\}$ be a basis for $\G{g}$, so that; $\{\xi_1,\ldots,\xi_m\}$ is a basis for $\G{m}$, and $\{\xi_{m+1},\ldots,\xi_n\}$ is a basis for $\G{h}$. Let also
\[
\widetilde{U}(\G{m}) = \Big\{\sum_{r_1,\ldots, r_m \geq 0}\,\a_{r_1,\ldots,r_m}\,\xi_1^{r_1}\ldots \xi_m^{r_m}\mid \a_{r_1,\ldots,r_m} \in k\Big\}
\]
and
\[
\widetilde{U}(\G{h}) = \Big\{\sum_{r_{m+1},\ldots, r_n \geq 0}\,\a_{r_{m+1},\ldots,r_n}\,\xi_{m+1}^{r_{m+1}}\ldots \xi_n^{r_n}\mid \a_{r_{m+1},\ldots,r_n} \in k\Big\}.
\]
Then, 
\[
U(\G{m} {\,}_{\g\hspace{-0.1cm}}\bowtie_\t\G{h}) \cong \widetilde{U}(\G{m}) {\,}_{\g\hspace{-0.1cm}}\bowtie_\t \widetilde{U}(\G{h})
\]
as bialgebras.
\end{corollary}

\begin{proof}
It follows at once that both $\widetilde{U}(\G{m})$ and $\widetilde{U}(\G{h})$ are subcoalgebras of $U(\G{g})$. That is, we have 
\[
\xymatrix{
\widetilde{U}(\G{m}) \ar@{^{(}->}[r]_i  & U(\G{g})    &  \ar@{_{(}->}[l]^j_{\rm alg}  \widetilde{U}(\G{h}).
}
\]
The claim then follows from Proposition \ref{prop-universal-II}. 
\end{proof}

\section*{Acknowledgment}

OE and SS acknowledge the support by T\"UB\.ITAK (the Scientific and Technological Research Council of Turkey) through the project ``Matched pairs of Lagrangian and Hamiltonian Systems'' with the project number 117F426. PG is grateful to Alberto Elduque for the illuminating discussions on the earlier phases of this work.

\bibliographystyle{plain}
\bibliography{references}{}

\def\polhk#1{\setbox0=\hbox{#1}{\ooalign{\hidewidth
  \lower1.5ex\hbox{`}\hidewidth\crcr\unhbox0}}} \def\cprime{$'$}
  \def\cprime{$'$} \def\cprime{$'$} \def\cprime{$'$} \def\cprime{$'$}
  \def\cprime{$'$} \def\cprime{$'$} \def\cprime{$'$} \def\cprime{$'$}
  \def\cprime{$'$} \def\cprime{$'$} \def\Dbar{\leavevmode\lower.6ex\hbox to
  0pt{\hskip-.23ex \accent"16\hss}D}
  \def\cfac#1{\ifmmode\setbox7\hbox{$\accent"5E#1$}\else
  \setbox7\hbox{\accent"5E#1}\penalty 10000\relax\fi\raise 1\ht7
  \hbox{\lower1.15ex\hbox to 1\wd7{\hss\accent"13\hss}}\penalty 10000
  \hskip-1\wd7\penalty 10000\box7}
  \def\cftil#1{\ifmmode\setbox7\hbox{$\accent"5E#1$}\else
  \setbox7\hbox{\accent"5E#1}\penalty 10000\relax\fi\raise 1\ht7
  \hbox{\lower1.15ex\hbox to 1\wd7{\hss\accent"7E\hss}}\penalty 10000
  \hskip-1\wd7\penalty 10000\box7} \def\cprime{$'$}
\begin{thebibliography}{10}

\bibitem{AgorMili11}
A.~L. Agore and G.~Militaru.
\newblock Extending structures {II}: {T}he quantum version.
\newblock {\em J. Algebra}, 336:321--341, 2011.

\bibitem{AgorMili13}
A.~L. Agore and G.~Militaru.
\newblock Unified products for {L}eibniz algebras. {A}pplications.
\newblock {\em Linear Algebra Appl.}, 439(9):2609--2633, 2013.

\bibitem{AgorMili14-III}
A.~L. Agore and G.~Militaru.
\newblock Bicrossed products, matched pair deformations and the factorization
  index for {L}ie algebras.
\newblock {\em SIGMA Symmetry Integrability Geom. Methods Appl.}, 10:Paper 065,
  16, 2014.

\bibitem{AgorMili14}
A.~L. Agore and G.~Militaru.
\newblock Extending structures for {L}ie algebras.
\newblock {\em Monatsh. Math.}, 174(2):169--193, 2014.

\bibitem{AgorMili14-II}
A.~L. Agore and G.~Militaru.
\newblock Extending structures {I}: the level of groups.
\newblock {\em Algebr. Represent. Theory}, 17(3):831--848, 2014.

\bibitem{AgorMili15}
A.~L. Agore and G.~Militaru.
\newblock The global extension problem, crossed products and co-flag
  non-commutative {P}oisson algebras.
\newblock {\em J. Algebra}, 426:1--31, 2015.

\bibitem{AgorMili15-II}
A.~L. Agore and G.~Militaru.
\newblock Jacobi and {P}oisson algebras.
\newblock {\em J. Noncommut. Geom.}, 9(4):1295--1342, 2015.

\bibitem{AgoreMilitaru-book}
A.~L. Agore and G.~Militaru.
\newblock {\em Extending {S}tructures: {F}undamentals and {A}pplications}.
\newblock CRC Press, Taylor \& Francis Group, 2019.
\newblock Chapman \& Hall/CRC Monographs and Research Notes in Mathematics.

\bibitem{BeggMaji90}
E.~Beggs and S.~Majid.
\newblock Matched pairs of topological {L}ie algebras corresponding to {L}ie
  bialgebra structures on {${\rm diff}(S^1)$} and {${\rm diff}({\bf R})$}.
\newblock {\em Ann. Inst. H. Poincar\'{e} Phys. Th\'{e}or.}, 53(1):15--34,
  1990.

\bibitem{BeggGoulMaji96}
E.~J. Beggs, J.~D. Gould, and S.~Majid.
\newblock Finite group factorizations and braiding.
\newblock {\em J. Algebra}, 181(1):112--151, 1996.

\bibitem{BespDrab99}
Y.~Bespalov and B.~Drabant.
\newblock Cross product bialgebras. {I}.
\newblock {\em J. Algebra}, 219(2):466--505, 1999.

\bibitem{BespDrab01}
Y.~Bespalov and B.~Drabant.
\newblock Cross product bialgebras. {II}.
\newblock {\em J. Algebra}, 240(2):445--504, 2001.

\bibitem{Brze97-II}
T.~Brzezi\'{n}ski.
\newblock Crossed products by a coalgebra.
\newblock {\em Comm. Algebra}, 25(11):3551--3575, 1997.

\bibitem{BrzeHaja99}
T.~Brzezi{\'n}ski and P.~M. Hajac.
\newblock Coalgebra extensions and algebra coextensions of {G}alois type.
\newblock {\em Comm. Algebra}, 27(3):1347--1367, 1999.

\bibitem{BrzeHaja09}
T.~Brzezi\'nski and P.~M. Hajac.
\newblock {G}alois-type extensions and equivariant projectivity,
  arXiv:0901.0141, (2009).

\bibitem{BrzeMaji98}
T.~Brzezi{\'n}ski and S.~Majid.
\newblock Coalgebra bundles.
\newblock {\em Comm. Math. Phys.}, 191(2):467--492, 1998.

\bibitem{ChevEile48}
C.~Chevalley and S.~Eilenberg.
\newblock Cohomology theory of {L}ie groups and {L}ie algebras.
\newblock {\em Trans. Amer. Math. Soc.}, 63:85--124, 1948.

\bibitem{Dixmier-book}
J.~Dixmier.
\newblock {\em Enveloping algebras}, volume~11 of {\em Graduate Studies in
  Mathematics}.
\newblock American Mathematical Society, Providence, RI, 1996.
\newblock Revised reprint of the 1977 translation.

\bibitem{EsenGrmeGumrPave19}
O.~Esen, M.~Grmela, H.~G\"{u}mral, and M.~Pavelka.
\newblock Lifts of symmetric tensors: fluids, plasma, and {G}rad hierarchy.
\newblock {\em Entropy}, 21(9):Paper No. 907, 33, 2019.

\bibitem{EsenSutlKude21}
O.~Esen, M.~Kudeyt, and S.~S\"{u}tl\"{u}.
\newblock Second order {L}agrangian dynamics on double cross product groups.
\newblock {\em J. Geom. Phys.}, 159:103934, 18, 2021.

\bibitem{EsenSard2021}
O.~Esen, C.~S. Mu{\~n}oz, and M.~Zajac.
\newblock Matched pair analysis of euler-poincar$\backslash$'$\{$e$\}$ flow on
  hamiltonian vector fields.
\newblock {\em arXiv preprint arXiv:2103.04401}, 2021.

\bibitem{EsenSutl16}
O.~Esen and S.~S\"{u}tl\"{u}.
\newblock Hamiltonian dynamics on matched pairs.
\newblock {\em Int. J. Geom. Methods Mod. Phys.}, 13(10):1650128, 24, 2016.

\bibitem{EsenSutl17}
O.~Esen and S.~S\"utl\"u.
\newblock Lagrangian dynamics on matched pairs.
\newblock {\em J. Geom. Phys.}, 111:142--157, 2017.

\bibitem{EsenSutl21}
O.~Esen and S.~S\"{u}tl\"{u}.
\newblock Discrete dynamical systems over double cross-product {L}ie groupoids.
\newblock {\em Int. J. Geom. Methods Mod. Phys.}, 18(4):2150057, 40, 2021.

\bibitem{EsenSutl20}
O.~Esen and S.~Sütlü.
\newblock Matched pair analysis of the {V}lasov plasma, arXiv:2004.12595v2,
  (2020).

\bibitem{Fuks-book}
D.~B. Fuks.
\newblock {\em Cohomology of infinite-dimensional {L}ie algebras}.
\newblock Contemporary Soviet Mathematics. Consultants Bureau, New York, 1986.
\newblock Translated from the Russian by A. B. Sosinski{\u\i}.

\bibitem{Gonc73}
L.~V. Gon\v{c}arova.
\newblock Cohomology of {L}ie algebras of formal vector fields on the line.
\newblock {\em Funkcional. Anal. i Prilo\v{z}en.}, 7(3):33--44, 1973.

\bibitem{HadfMaji07}
T.~Hadfield and S.~Majid.
\newblock Bicrossproduct approach to the {C}onnes-{M}oscovici {H}opf algebra.
\newblock {\em J. Algebra}, 312(1):228--256, 2007.

\bibitem{HochSerr53-II}
G.~Hochschild and J.-P. Serre.
\newblock Cohomology of group extensions.
\newblock {\em Trans. Amer. Math. Soc.}, 74:110--134, 1953.

\bibitem{HochSerr53}
G.~Hochschild and J.~P. Serre.
\newblock Cohomology of {L}ie algebras.
\newblock {\em Ann. of Math. (2)}, 57:591--603, 1953.

\bibitem{Hold1895}
O.~H\"{o}lder.
\newblock Bildung zusammengesetzter {G}ruppen.
\newblock {\em Math. Ann.}, 46(3):321--422, 1895.

\bibitem{Kac68}
V.~G. Kac.
\newblock Simple irreducible graded {L}ie algebras of finite growth.
\newblock {\em Izv. Akad. Nauk SSSR Ser. Mat.}, 32:1323--1367, 1968.

\bibitem{LuWe90}
J.-H. Lu and A.~Weinstein.
\newblock Poisson {L}ie groups, dressing transformations, and {B}ruhat
  decompositions.
\newblock {\em J. Differential Geom.}, 31(2):501--526, 1990.

\bibitem{Majid-thesis}
S.~Majid.
\newblock {\em Noncommutative-geometric groups by a bicrossproduct
  construction: {H}opf algebras at the {P}lanck scale}.
\newblock ProQuest LLC, Ann Arbor, MI, 1988.
\newblock Thesis (Ph.D.)--Harvard University.

\bibitem{Maji90-II}
S.~Majid.
\newblock Matched pairs of {L}ie groups associated to solutions of the
  {Y}ang-{B}axter equations.
\newblock {\em Pacific J. Math.}, 141(2):311--332, 1990.

\bibitem{Maji90}
S.~Majid.
\newblock Physics for algebraists: noncommutative and noncocommutative {H}opf
  algebras by a bicrossproduct construction.
\newblock {\em J. Algebra}, 130(1):17--64, 1990.

\bibitem{Majid-book}
S.~Majid.
\newblock {\em Foundations of quantum group theory}.
\newblock Cambridge University Press, Cambridge, 1995.

\bibitem{Majid99}
S.~Majid.
\newblock Double-bosonization of braided groups and the construction of
  {$U_q(g)$}.
\newblock {\em Math. Proc. Cambridge Philos. Soc.}, 125(1):151--192, 1999.

\bibitem{MajiRueg94}
S.~Majid and H.~Ruegg.
\newblock Bicrossproduct structure of {$\kappa$}-{P}oincar\'e group and
  non-commutative geometry.
\newblock {\em Phys. Lett. B}, 334(3-4):348--354, 1994.

\bibitem{Mood67}
R.~V. Moody.
\newblock Lie algebras associated with generalized {C}artan matrices.
\newblock {\em Bull. Amer. Math. Soc.}, 73:217--221, 1967.

\bibitem{MoscRang07}
H.~Moscovici and B.~Rangipour.
\newblock Cyclic cohomology of {H}opf algebras of transverse symmetries in
  codimension 1.
\newblock {\em Adv. Math.}, 210(1):323--374, 2007.

\bibitem{MoscRang09}
H.~Moscovici and B.~Rangipour.
\newblock Hopf algebras of primitive {L}ie pseudogroups and {H}opf cyclic
  cohomology.
\newblock {\em Adv. Math.}, 220(3):706--790, 2009.

\bibitem{Olve96}
P.~Olver.
\newblock Non-{A}ssociative {L}ocal {L}ie {G}roups.
\newblock {\em J. Lie Theory}, 6(1):23--51, 1996.

\bibitem{Ore37}
O.~Ore.
\newblock Structures and group theory. {I}.
\newblock {\em Duke Math. J.}, 3(2):149--174, 1937.

\bibitem{Ercu-book}
E.~Orta\c{c}gil.
\newblock {\em An {A}lternative {A}pproach to {L}ie {G}roups and {G}eometric
  {S}tructures}.
\newblock Oxford University Press, UK, 2018.

\bibitem{Sche25}
O.~Schreier.
\newblock \"{U}ber die {E}rweiterung von {G}ruppen. {II}.
\newblock {\em Abh. Math. Sem. Univ. Hamburg}, 4(1):321--346, 1925.

\bibitem{Sche26}
O.~Schreier.
\newblock \"{U}ber die {E}rweiterung von {G}ruppen {I}.
\newblock {\em Monatsh. Math. Phys.}, 34(1):165--180, 1926.

\bibitem{Sing70}
W.~M. Singer.
\newblock Extension theory for connected {H}opf algebras.
\newblock {\em Bull. Amer. Math. Soc.}, 76:1095--1099, 1970.

\bibitem{Sing72}
W.~M. Singer.
\newblock Extension theory for connected {H}opf algebras.
\newblock {\em J. Algebra}, 21:1--16, 1972.

\bibitem{Szep49}
J.~Sz\'{e}p.
\newblock \"{U}ber die als {P}rodukt zweier {U}ntergruppen darstellbaren
  endlichen {G}ruppen.
\newblock {\em Comment. Math. Helv.}, 22:31--33 (1948), 1949.

\bibitem{Szep50}
J.~Sz\'{e}p.
\newblock On the structure of groups which can be represented as the product of
  two subgroups.
\newblock {\em Acta Sci. Math. (Szeged)}, 12:57--61, 1950.

\bibitem{SzepRede50}
J.~Sz\'{e}p and L.~R\'{e}dei.
\newblock On factorisable groups.
\newblock {\em Acta Univ. Szeged. Sect. Sci. Math.}, 13:235--238, 1950.

\bibitem{Take81}
M.~Takeuchi.
\newblock Matched pairs of groups and bismash products of {H}opf algebras.
\newblock {\em Comm. Algebra}, 9(8):841--882, 1981.

\bibitem{Vira70}
M.~A. Virasoro.
\newblock Subsidiary conditions and ghosts in dual-resonance models.
\newblock {\em Phys. Rev. D.}, 10:2933--2936, 1970.

\bibitem{Zapp42}
G.~Zappa.
\newblock Sulla costruzione dei gruppi prodotto di due dati sottogruppi
  permutabili tra loro.
\newblock In {\em Atti {S}econdo {C}ongresso {U}n. {M}at. {I}tal., {B}ologna,
  1940}, pages 119--125. Edizioni Cremonense, Rome, 1942.

\end{thebibliography}

\end{document}